\documentclass[12pt,a4paper]{amsart}
\usepackage{mathrsfs}
\usepackage{amsfonts}
\usepackage{txfonts}

\usepackage{hyperref}
\usepackage{latexsym}
\usepackage{amssymb}

\newtheorem{theorem}{Theorem}[section]
\newtheorem{lemma}[theorem]{Lemma}
\newtheorem{corollary}[theorem]{Corollary}
\newtheorem{proposition}[theorem]{Proposition}

\theoremstyle{definition}
\newtheorem{definition}[theorem]{Definition}

\newtheorem{remark}[theorem]{Remark}

\newcommand{\R}{\mathbb{R}}
\newcommand{\N}{\mathbb{N}}
\newcommand{\eps}{\varepsilon}

\numberwithin{equation}{section}

\begin{document}

\title[Representation  of MNCs and its applications]
{{\bf Representation  of measures of noncompactness and its applications related to an initial-value problem in Banach spaces}}

\author{Xiaoling Chen, \;Lixin Cheng$^\dag$}

\address{Xiaoling Chen: School of Mathematical Sciences, Xiamen University,
Xiamen, 361005, China}
\email{30128299@qq.com\;\;(X. Chen)}
\address{ Lixin Cheng:  School of Mathematical Sciences, Xiamen University,
Xiamen, 361005, China}
\email{lxcheng@xmu.edu.cn\;\;(L. Cheng)}
\thanks{$^\dag$ The corresponding author; support
by NSFC, grant 11731010}

\date{25/01/2021}

\begin{abstract} The purpose of  this paper is devoted to  studying  representation  of   measures of non generalized compactness, in particular, measures of  noncompactness, of non-weak compactness, and of non-super weak compactness, etc,  defined on Banach spaces and its applications. With the aid of a three-time order preserving embedding theorem,  we show that for every Banach space $X$, there exist a Banach function space $C(K)$ for some compact Hausdorff space $K$, and an order-preserving affine mapping $\mathbb T$ from the super space $\mathscr B$ of all nonempty bounded subsets of $X$ endowed with the Hausdorff metric to the positive cone $C(K)^+$ of $C(K)$ such that for every convex measure, in particular, regular measure, homogeneous measure, sublinear measure of non generalized compactness  $\mu$ on $X$, there is a convex function $\digamma$ on the cone $V=\mathbb T(\mathscr B)$ which is Lipschitzian on each bounded set of $V$ such that
\[\digamma(\mathbb T(B))=\mu(B),\;\;\forall\;B\in\mathscr B.\]
As its applications, we show a class of basic integral inequalities related to an initial-value problem in Banach spaces, and prove a solvability result of the initial-value problem, which is an  extension of some classical results due to Goebel, Rzymowski,  and Bana\'{s}.

\end{abstract}



\maketitle

\section{Introduction}

There are  three goals of this paper: (1) to establish representation theorem of convex  measures of noncompactness (convex MNCs, for simplicity), in particular,  regular MNCs, homogeneous MNCs, sublinear MNCs  and of their generalizations, including convex measures of non-weak compactness (convex MNWCs), of non-super weak compactness (convex MNSWCs), of non-Radon-Nikod\'ym property (convex MNRNPs) and of non-Asplundness (convex MNAs); (2) to establish a class of  basic integral inequalities related to an initial-value problem in Banach spaces; and as their application,   (3) to discuss solvability of the initial-value problem.

The letter $X$ will always be an infinite dimensional real Banach space, and $X^*$ its dual. $B_X$ stands for the closed unit ball of $X$, and $B(x,r)$ for the closed ball centered at $x$ with radius $r$. $\mathscr{B}(X)$ denotes the collection of all nonempty bounded subsets of $X$ endowed with the Hausdorff metric. $\Omega$ stands for the closed unit ball  $B_{X^*}$ of $X^*$ endowed with the norm topology, and $C_b(\Omega)$ is the Banach space of all continuous bounded functions on $\Omega$ endowed with the sup-norm. For a subset $A\subset X$, $\overline{A}$ stands for the norm closure of $A$, and ${\rm co}(A)$ for the convex hull  of $A$.

The main results of this paper consist of the following three parts.

{\bf Part I. \;A representation theorem of convex MNGCs.}   For every  Banach space $X$,  there exist a Banach function space $C(K)$ for some compact Hausdorff space $K$, and  an order-preserving affine 1-Lipschitz mapping $\mathbb T: \mathscr B(X)\rightarrow C(K)^+$ such that  for every `` convex measure of non generalized compactness (convex MNGC)" (in particular, convex MNC, convex MNWC, convex MNSWC, convex MNRNP, and convex MNA) $\mu$ defined on $X$, there is a  continuous  convex function $\digamma$ on the positive cone $V\equiv\mathbb T\Big(\mathscr B(X)\Big)$ , which is monotone increasing in the order of set inclusion and Lipschitzian on each bounded subset of $V$ satisfying
\begin{equation}
\digamma(\mathbb T(B))=\mu(B),\;\forall B\in\mathscr B(X).
\end{equation}
If, in addition, $\mu$ is a sublinear measure of non generalized compactness, then $\digamma$ is a $\mu(B_X)$-Lipschitian sublinear functional on $V$.

{\bf Part II. A class of basic integral inequalities.} For every nonempty subset $G\subset L_1([0,a],X)$ of integrable $X$-valued  functions with $\psi(t)\equiv\sup_{g\in G}\|g(t)\|$  integrable on $[0,a]$ such that the mapping $\mathbb J_G: [0,a]\rightarrow C_b(\Omega)$ defined for $t\in[0,a]$ by
\begin{equation}
\mathbb J_G(t)(\omega)=\sup_{g\in G}\langle\omega,g(t)\rangle\equiv\sigma_{G(t)}(\omega), \;\omega\in\Omega\equiv B_{X^*},
\end{equation}
is strongly (Lebesgue-Bochner) measurable, then for every  convex measure of non generalized compactness (convex MNGC) $\mu$ defined on $X$, we have
\begin{equation}
\mu\Big\{\int_0^\tau G(s)ds\Big\}\leq\frac{1}\tau\int_0^\tau\mu\Big\{\tau G(s)\Big\}ds,\;\forall 0<\tau\leq a;
\end{equation}
in particular, if $\mu$ is a  sublinear measure of non generalized compactness (sublinear MNGC), or, $\tau\leq1$, then
\begin{equation}
\mu\Big\{\int_0^\tau G(s)ds\Big\}\leq\int_0^\tau\mu\Big\{G(s)\Big\}ds.
\end{equation}

{\bf Part III. Solvability of a Cauchy problem.}  As an application of the results mentioned above, we consider solvability of the following initial value problem
\begin{equation} \left\{\begin{array}{cc}
                    x'(t)=f(t, x), & a\geq t>0; \\
                    x(0)=x_0~~~~~~~~
                  \end{array}
\right.
\end{equation}
and give an extension of  a classical solvability result of the problem  due to K. Goebel and W. Rzymowski \cite{goe} (see, also, \cite{rz}), and due to J. Bana\'{s},  K. Goebel \cite{bag}.\\

The study of MNCs and of their applications has continued for 90 years since the first MNC (Kuratowski's MNC in the present terminology) was introduced by  K. Kuratowski \cite{kur} in 1930. It has been shown that the theory of measures of noncompactness  was  used  in  a wide variety of topics in nonlinear analysis. %(see, for instance , Akhmerov et al.[8],  Appell[9],  Ayerbe Toledano et al.[10],  Bana$\acute{s}$ and Goebel [4,11], Djebali et al.[12]  and Meskhi [13]).
Roughly speaking, an MNC $\mu$ is a   nonnegative function defined on the family $\mathscr B(X)$ consisting of all nonempty bounded subsets of a complete metric space, in particular, a Banach space $X$ and satisfies some specific properties such as  non-decreasing monotonicity in the order of  the set inclusion, and the (most important) noncompactness that $\mu(B)=0$ if and only if $B$ is relatively compact in $X$.

The first MNC $\alpha$ was introduced by  K. Kuratowski \cite[1930]{kur} for every $B\in\mathscr B(X)$ of a complete metric space $X$:
\begin{equation}
\alpha(B)=\inf\{d>0: B\subset\cup_{j\in F} E_j\subset X, \;d(E_j)\leq d, F\subset\N,  F^\sharp<\infty\},
\end{equation}
where $d(E_j)$ denotes  the diameter of $E_j$, $F^\sharp$ denotes the cardinality of the set $F\subset\mathbb N$. The earliest successful application of Kuratowski's MNC  was applied in the fixed point theory.
In 1955, G. Darbo \cite{da} extended  the Schauder fixed point theorem to noncompact mappings,  named set-contractive operators. Since then, the study of MNCs and of their applications has become an active research area, and various MNCs have appeared. Among many other MNCs, the Hausdorff MNC $\beta$ is another widely used MNC,   which was introduced by Gohberg,  Gol'den\'{s}shte\'{i}n and Markus \cite{goh} in 1957. It is defined for $B\in\mathscr B(X)$ by
\begin{equation}
\beta(B)=\inf\{r>0: B\subset\cup_{x\in F}B(x,r),\;F\subset X,\;F^\sharp<\infty\},\end{equation}
where $B(x,r)$ denotes the closed ball centered at $x\in X$ with radius $r$.

It is easy to observe that if $\mu$ is either  Kuratowski's MNC $\alpha$ or the Hausdorff MNC $\beta$, then it satisfies the following three conditions.

(1) [{\bf Noncompactness}] $B\in\mathscr B(X),\;\mu(B)=0 \Longleftrightarrow B$ is relatively compact;

(2) [{\bf Monotonicity}] $A, B\in\mathscr B(X)\; {\rm with\;}A\supset B\;\Longrightarrow \mu(A)\geq\mu(B);$

(3) [{\bf Order preserving}] $A, B\in\mathscr B(X) \Longrightarrow\;\mu(A\cup B)=\mu(A)\vee\mu(B).$

%(4) $B\in\mathscr B(M) \Longrightarrow \mu(\overline{B})= \mu(B)$, where $\overline{B}$ denotes the closure of $B$.

\noindent
In particular, if $X$ is a Banach space over the scalar field $\mathbb F$,  then

(4) [{\bf Convexification invariance}] $B\in\mathscr B(X)\Longrightarrow \mu({\rm co}{(B)})= \mu(B)$;

(5) [{\bf Absolute homogeneity}] $B\in\mathscr B(X), k\in\mathbb F \Longrightarrow \mu(kB)= |k|\mu(B)$;

(6) [{\bf Subadditivity}] $A, B\in\mathscr B(X) \Longrightarrow\;\mu(A+B)\leq\mu(A)+\mu(B).$

\begin{definition} Let $X$ be a Banach space and $\mathscr B(X)$ be the family of all nonempty bounded subsets of $X$,  $\mu: \mathscr B(X)\rightarrow \R^+$ be  a real-valued function.

i) (\cite{bag}) $\mu$  is said to be a {\bf regular MNC} on $X$ provided it satisfies all the six properties ((1)-(6)) above.

ii) (\cite{mal}) $\mu$ is called a {\bf homogeneous  MNC} on $X$ if it satisfies (1), (2), (4)-(6).

iii) We say that $\mu$ is a {\bf sublinear MNC} on $X$ if it satisfies Properties (1), (2), (4), (6) and

(7) [{\bf Positive homogeneity}] $B\in\mathscr B(X), k\geq0 \Longrightarrow \mu(kB)= k\mu(B)$.

iv)  $\mu$ is a {\bf convex MNC} on $X$ if it satisfies Properties (1), (2), (4) and

(8) [{\bf Convexity}] $\forall A, B\in \mathscr B(X)\;{\rm and\;}0\leq\lambda\leq1,$ \[\mu(\lambda A+(1-\lambda)B)\leq\lambda\mu(A)+(1-\lambda)\mu(B).\]
\end{definition}
\begin{remark}
 Clearly, a regular MNC is a homogeneous MNC, a homogeneous MNC is a sublinear MNC, and a convex MNC.
 The notion of homogeneous MNC is called `` sublinear full MNC" by Bana\'{s} and Geobel \cite{bag}.

\end{remark}

In this paper, we always assume that the space $X$ in question is a real Banach space.  We should mention here that  a Banach space setting for the study of MNGCs would lose no generality, because that every metric space is isometric to a subset of a Banach space \cite[Lemma 1.1]{ben}. Besides, the concept is easily to be generalized by Banach space theory, for example, it has been generalized in various ways such as MNWC (see, for instance, \cite{cas},\cite{de},\cite{ha}, \cite{ha1});  of MNRNP \cite{z}; or, more general, measures of non-property $\mathfrak{D}$ \cite{ccs}.\\

Recall that $\mathscr B(X)$ denotes the collection of all nonempty bounded subsets of a Banach space $X$ endowed with the Hausdorff metric $d_H$ defined for $A, B\in\mathscr B(X)$ by
\[d_H(A,B)=\inf\{d>0: A\subset B+dB_X, B\subset A+dB_X\},\]
where $B_X$ is the closed unit ball of $X$.

We use $\mathscr C(X) (\subset\mathscr B(X)) $ to denote the collection of all nonempty bounded closed convex subsets of a Banach space $X$ (endowed with the Hausdorff metric) and with the following set addition $\bigoplus$:
\[A\bigoplus B=\overline{A+B},\;\;\forall\;A, B\in\mathscr C(X),\]  and the usual scalar multiplication of sets:
\[\lambda A=\{ka: a\in A\},\;\forall\;A\in\mathscr C(X)\;{\rm and\;}k\in\R.\]
If we put \[\||A\||=\sup_{a\in A}\|a\|,\;\;\forall\;A\in\mathscr C(X),\]
then it becomes a complete normed semigroup  with respect to  the Hausdorff metric \cite{ccs}. (See, also, Section 2 in detail.)

\begin{definition}[\cite{ccs}] Let $X$ be a Banach space and $\mathscr D\equiv\mathscr D(X)$ be a closed subsemigroup of $\mathscr C(X)$. It is said to be fundamental provided it satisfies

(1) $\bigcup\{D\in \mathscr{D}\}=X$;

(2) $A,B\in \mathscr{D} \;{\rm entails }\;{\rm\overline{co}}(\pm A\cup\pm B)\in \mathscr{D}$;

(3) $\emptyset\neq B\subset D\in  \mathscr{D}  \;{\rm implies }\;{\rm\overline{co}}(B)\in \mathscr D$.
\end{definition}

There are many possibilities. But the following are most interesting to us:
i) $\mathscr D=\mathscr K(X)$, the subsemigroup of all nonempty  compact convex subsets of $X$; ii) $\mathscr D=\mathscr W(X)$,  all nonempty  weakly compact convex subsets of $X$; iii) $\mathscr D=\mathscr R(X)$,  all nonempty closed bounded convex subsets $D$ of $X$  admitting the Radon-Nikod\'ym property; iv) $\mathscr D=\mathscr A(X)$, all nonempty bounded closed  convex Asplund subsets $D$ of $X$; and v) $\mathscr D={\rm sup\text{-}}\mathscr W(X)$, all nonempty convex super weakly compact subsets $D$ of $X$.

\begin{definition} Let $X$ be a Banach space, $\mathscr D\equiv\mathscr D(X)$ be a  fundamental sub-semigroup of $\mathscr C(X)$,  and $\mu: \mathscr B(X)\rightarrow \R^+$ be a real-valued function. It is said to be a regular measure of non-property $\mathfrak D$ (regular MNP$\mathfrak D$) provided

(1) [{\bf Non-property $\mathfrak D$}] $B\in\mathscr B(X),\;\mu(B)=0 \Longleftrightarrow {\rm\overline{co}}(B)\in\mathscr D$;

(2) [{\bf Monotonicity}] $A, B\in\mathscr B(X)\; {\rm with\;}A\supset B\;\Longrightarrow \mu(A)\geq\mu(B);$

(3) [{\bf Order preserving}] $A, B\in\mathscr B(X) \Longrightarrow\;\mu(A\cup B)=\mu(A)\vee\mu(B).$

%(4) $B\in\mathscr B(M) \Longrightarrow \mu(\overline{B})= \mu(B)$, where $\overline{B}$ denotes the closure of $B$.

(4) [{\bf Convexification invariance}] $B\in\mathscr B(X)\Longrightarrow \mu({\rm co}{(B)})= \mu(B)$;

(5) [{\bf Absolute homogeneity}] $B\in\mathscr B(X), k\in\mathbb F \Longrightarrow \mu(kB)= |k|\mu(B)$;

(6) [{\bf Subadditivety}] $A, B\in\mathscr B(X) \Longrightarrow\;\mu(A+B)\leq\mu(A)+\mu(B).$

In particular, a regular MNP$\mathfrak D$ is just a regular MNC if $\mathscr D=\mathscr K$; and a regular MNWC if $\mathscr D=\mathscr W$.
\end{definition}
\begin{definition} Let $X$ be a Banach space, $\mathscr D\equiv\mathscr D(X)$ be a  fundamental subsemigroup of $\mathscr C(X)$ and $\mu: \mathscr B(X)\rightarrow \R^+$ be a real-valued function.

i) $\mu$ is called a {\bf homogeneous measure of non\text{-}property $\mathfrak D$} (homogeneous MNP$\mathfrak D$) on $X$ if it satisfies (1), (2), (4)-(6) in Definition 1.4.

ii) $\mu$ is called a {\bf sublinear measure of non\text{-}property $\mathfrak D$} (sublinear MNP$\mathfrak D$) on $X$ if it satisfies (1), (2), (4), (6) in Definition 1.4 and

(7) [{\bf Positive homogeneity}] $B\in\mathscr B(X), k\geq0 \Longrightarrow \mu(kB)= k\mu(B)$.

iii) We say that $\mu$ is a {\bf convex measure of non\text{-}property $\mathfrak D$} (convex MNP$\mathfrak D$) on $X$ if it satisfies Properties (1), (2), (4) in Definition 1.4,  and

(8) [{\bf Convexity}] $\mu(\lambda A+(1-\lambda)B)\leq\lambda\mu(A)+(1-\lambda)\mu(B),\;\forall A, B\in \mathscr B(X)\;{\rm and\;}0\leq\lambda\leq1.$ Therefore, a sublinear measure of noncompactness is a convex one.

In particular, a convex (resp., homogeneous, sublinear) MNP$\mathfrak D$ is just a convex (resp., homogeneous, sublinear) MNC if $\mathscr D=\mathscr K$; and a convex (resp., homogeneous,  sublinear) MNWC if $\mathscr D=\mathscr W$.
\end{definition}

A great deal of effort has been expended in studying representation problem of some concrete MNCs, especially, the Hausdorff MNC
on typical Banach spaces such as $C(K)$ and $L_p$. See, for example, \cite{ak, ay, bag, bam}.  Nevertheless,  the representation problem of abstract MNCs on a general Banach space has never been studied before and poses significant difficulties. With the help of a three-time-order-preserving embedding theorem \cite{ccs} established in 2018 (which will be introduced in the next section) and by arguments in convex analysis, especially, in subdifferentiability of convex functions defined on convex sets with empty interiors, in the first part of this paper, we show the representation theorem previously mentioned in the first section for convex MNCs.\\

Many mathematicians have made great efforts to expect the following type of integral inequalities with various assumptions:
\begin{equation}\mu\Big\{\int_0^ax_n(\omega)d\omega: n\in\mathbb N\Big\}\leq \int_0^a\mu\Big\{x_n(\omega):n\in\mathbb N\Big\}d\omega,\end{equation}
where $\{x_n\}$ is a sequence of continuous $X$-valued functions in $C([0,a], X)$, and  $\mu$ is the Hausdorff MNC, or, the Kuratowski's MNC. Such type of integral inequalities arise from the initial value problem (1.5) in Banach spaces.
Among the many  extra conditions which are sufficient for the solvability of the equation, one of the most important types is under hypothesises in terms of MNCs. Indeed, a standard proof for the existence of solutions employs the Darbo fixed point theorem in connection with  the following Picard-Lindel\"{o}f operator
\begin{equation}Ax(t)=x_0+\int_0^t f(s,x(s))ds, \ \  x\in B,\end{equation}
 where $B$ is a bounded subset of $C([0,a], X)$, and $f$ is called the Carath\'{e}odory function. One wants to obtain that if the  Carath\'{e}odory function satisfies that $f(t,\cdot)$ is condensing for all $t\in[0,a]$, then the operator $A: C([0,a], X)\rightarrow  C([0,a], X)$ defined by (1.9) is also condensing. More precisely, one may expect that for  $\{x_n\}\subset B$,
\begin{equation}\mu\Big\{\int_0^t f(s,x_n(s))ds: n\geq1\Big\}\leq \int_0^t \mu\Big\{f(s,x_n(s)): n\geq1\Big\}ds\end{equation}
 Once we have established an estimate of the form (1.10), there is not much left to do to prove the existence of solutions for (1.5).\\

 Some remarkable contributions have been made by  K. Goebel and  W. Rzymowdski \cite{goe},  J. Bana\'{s} and K. Goebel \cite{bag}, M\"{o}nch\cite{mo},  M\"{o}nch and G.-F. von Harten \cite{mo2}, H.P. Heinz \cite{hei},  M. Kunze and G. Schl\"{u}chtermann \cite{kun}. Nevertheless,
it still comes as a surprise that the estimate (1.8) has been proved only under some very restrictive assumptions in the existing literatures.  While a number of counterexamples constructed by Heinz \cite{hei} show that  it is quite complicated and  poses significant difficulties to present appropriate hypothesises to guarantee (1.8).
In 1970, Goebel and Rzymowdski \cite{goe} first showed that (1.8) holds for the Hausdorff MNC $\beta$, with the assumptions that $\{x_n\}\subset C([0,a],X)$  is  bounded and equi-continuous. In 1980, Bana\'{s} and Goebel  \cite{bag} further proved (1.8) holds again under the same assumptions but one can substitute a  sublinear  MNC  $\mu$  for $\beta$.   %In 1977,  Deimling \cite{dei} proved (1.10) with the following assumptions: a)  $\mu=\beta$, the Hausdorff measure of noncompactness;b)  $f$ is uniformly continuous on $[0,a]\times X$; and c)  $\{x_n\}\subset C([0,a];X)$ is a sequence of approximate solutions of (1.11).
%Motivated by Deimling \cite{dei},
M\"{o}nch \cite[1980]{mo}, M\"{o}nch and G.-F. von Harten \cite[1982]{mo2} proved (1.8) with the  assumptions a)  $\mu=\beta$,  b) $X$ is a weakly compactly generated space (in particular, a separable Banach space) and c)  $\{x_n\}\subset C([0,a];X)$ and there is $\psi\in L_1([0,a])$ such that $\sup_n\|x_n(t)\|\leq\psi(t)$ a.e. In 1998, Kunze and  Schl\"{u}chtermann \cite{kun} showed that (1.8) holds for  a  Grothendieck measure $\mu$ assuming that $X$ is strongly generated by a Grothendieck class.  But for an arbitrary Banach space, one has to insert the factor 2 in the right-hand side of (1.10),  i.e.
\begin{equation}\nonumber\beta\Big\{\int_0^t x_n(s)ds:n\geq 1\Big\}\leq 2\int_0^t \beta\Big\{x_n(s):n\geq 1\Big\}ds.\end{equation}
See, for instance, \cite{bo}.

In the second part of this paper, applying the representation theorem established in the first part, we show the result presented in Part II.

We also show that the mapping $\mathbb J_G$ is always weakly measurable if $G\subset L_1(I,X)$ is separable and the sup-envelop of $\{\|g\|: g\in G\}$ is again integrable. It is strongly measurable if, in addition, $G$ satisfies one of the following three conditions: i) $G\subset C([0,a],X)$ is equi-continuous;  ii) $G$ is separable and equi-regulated; and iii) $G$ is uniformly measurable.\\

The solvability question of the initial value problem (1.5) in infinite dimensional Banach spaces is important because  Peano's existence theorem of the problem (1.5) in finite dimensional spaces  is not valid in infinite dimensional Banach spaces. A number of mathematicians have made a series of contributions to the problem.
 In 1970, Goebel and Rzymowdski \cite{goe} using
the fixed point theorem proved that (1.5) has a solution with the following assumptions:
a) $f$ is bounded and uniformly continuous on $[0,a]\times B(x_0,r)$; and
b) for the Hausdorff MNC  $\beta$ defined on $X$,  $\beta(f(t,B))\leq w(t,\beta(B))$ for almost all $t\in [0,a]$, where $w$ is a Kamke function.
In 1977, Deimling \cite{dei} by means of approximate solutions proved that (1.5) has a solution under the following assumptions: a) $f$ is bounded and uniformly continuous on $[0,a]\times B(x_0,r)$; and c) for Kuratowski's MNC $\alpha$ defined on $X$, $\alpha(f(t,B))\leq L\alpha(B)$ for some $L>0$ and for all $t\in [0,a]$. In 1980,  Bana\'{s} and Goebel  \cite{bag} showed that (1.5) has a solution assuming a) $f$ is bounded and uniformly continuous on $[0,a]\times B(x_0,r)$; and d) $\mu(f(t,B))\leq w(t,\mu(B))$ for almost all $t\in [0,a]$, where $w$ is a Kamke function and $\mu$ is a (symmetric) sublinear MNC.   Motivated by Deimling \cite{dei}, M\"{o}nch and G.-F. von Harten \cite[1982]{mo2} proved that (1.5) has a solution under the following assumptions that $X$ is a weakly compactly generated space, $f$ is continuous, and  there is a Kamke function $w$ such that $\beta(f(t,B))\leq w(t,\beta(B))$ for almost all $t\in [0,a]$.

As an application of the representation theorem and the integral inequalities of presented in Parts I and II,  in the last part of this paper, we consider solvability of the  initial value problem (1.5),  and extend some classical results due to Goebel and Rzymowski \cite{goe}, Bana\'{s} and Goebel \cite{bag}, and Deimling \cite{dei}.

\section{Preliminaries}
In this section, we will introduce the three-time order preserving embedding theorem established in \cite{ccs}. This theorem will play an essential rule in the study of representation problem of measures of non generalized compactness of this paper. It was proven  by this theorem that``every infinite dimensional Banach space admits inequivalent regular measures of noncompactness" which is an affirmative answer to a question  proposed by Goebel in 1978 (see, \cite{ab}), also by  Mallet-Paret and Nussbaum in 2011 \cite{mal}. All notions and results presented in this section can be found in \cite{ccs}.

\begin{definition}Let $G$ be an abelian semigroup and let $\mathbb F\in\{\mathbb R,\mathbb C\}.$ $G$ is said to be a module if there are two operations $(x,y)\in G\times G\rightarrow x+y\in G$,
and $(\alpha, x)\in (\mathbb F\times G)\rightarrow \alpha x\in G$ satisfying
$$(\lambda\mu)g=\lambda(\mu g),\;\forall\lambda,\mu\in\mathbb R\;and\;g\in G;$$
$$\lambda(g_1+g_2)=\lambda g_1+\lambda g_2,\;\forall\lambda\in\mathbb F\;and\;g_1,g_2\in G;$$
and
$$1g=g\;\;and\;\;0g=0\;\;\forall g\in G. $$

A module $G$ endowed with a norm is called a normed semigroup.
\end{definition}
 For a (real) Banach space $X$ we denote by $\mathscr C(X)$ (or, simply, $\mathscr C$) the collection of all nonempty bounded closed convex sets of $X$ with respect to the following  linear operations:

(a)\;\;  $A\oplus B=\overline{\{a+b:a\in A, b\in B\}},$ for all $A, B\in\mathscr C$;

(b)\;\;  $kC=\{kc: c\in C\}$ for all $k\in\mathbb F$ and $C\in\mathscr C$.\\
We further define the following function $\||\cdot\||$  for $C\in\mathscr C$ by
\begin{equation}
\||C\||=\sup_{c\in C}\|c\|.
\end{equation}
Then it is easy to check that $\||\cdot\||$ is a norm on $\mathscr C$, i.e. it satisfies

i) $\||C\||\geq0 $ for all $C\in\mathscr C$,  and $\||C\||=0\Longleftrightarrow C=\{0\}$;

ii) $\||kC\||=|k|\cdot\||C\||$, for all $C\in\mathscr C$ and $k\in\mathbb R$;

iii) $\||A\oplus B\||\leq\||A\||+\||B\||$, for all $A,\;B\in\mathscr C$.\\
Therefore, $(\mathscr C,\||\cdot\||)$ is a normed (real) semigroup.

If we equip the normed semigroup  $\mathscr C$  with the  Hausdorff metric $d_H$:
\begin{equation}
d_H(A,B)=\inf\big\{r>0: A\subset B+rB_X,\;B\subset A+rB_X\big\}, \;A, B\in\mathscr C,
\end{equation}
then \begin{equation}\||C\||=d_H(\{0\},C),\;\;\forall  C\in\mathscr C.\end{equation}
%and $\mathscr C=(\mathscr C,d_H)$ is complete.

We also use $\mathscr{K}(X)$  to denote the (complete) sub-semigroup of $\mathscr{C}(X)$ consisting of all nonempty compact  convex subsets.
\begin{definition}
Let $\Gamma_1, \Gamma_2$ be two  partially ordered sets, and $f:\Gamma_1\rightarrow \Gamma_2$ be a mapping.

(i) Then $f$  is said to be (resp., fully) order-preserving provided it is bijective
and satisfying  $f(x)\geq f(y)\in\Gamma_2$  if (resp., and only if) $x\geq y\in\Gamma_1$.

(ii) If both  $\Gamma_1$  and $\Gamma_2$ are modules, then  we call the mapping $f$ affine  if it satisfies $f(ax+by)=af(x)+bf(y)$ for all $a, b\geq0$ and  $x, y\in\Gamma_1$.
\end{definition}

For each nonempty subset $A\subset X$, we denote by $\sigma_A$, the support function of $A$ restricted to $\Omega\equiv  B_{X^*}$, i.e.
\begin{equation}
\sigma_A(x^*)=\sup_{a\in A}\langle x^*,a\rangle, \;\;{\rm for}\;x^*\in\Omega.
\end{equation}
 Let $\Omega$ be endowed with the norm topology of $X^*$, and
 $\mathfrak S(\Omega)$ the wedge of all continuous and bounded-$w^*\text{-}$ lower semicontinuous sublinear functions on $X^*$ but restricted to $\Omega$. Then it is a closed cone of
 $C_b(\Omega)$, the Banach space of all  continuous functions bounded on $\Omega$ endowed with the sup-norm. Let $J: \mathscr C(X)\rightarrow \mathfrak S(\Omega)$ be defined by
  \begin{equation}
  J(C)=\sigma_C,\;\;\forall\;C\in\mathscr C(X).
  \end{equation}
If we order the normed semigroup $\mathscr C(X)$ by set inclusion, i.e. $A\geq B$ if and only if $A\supset B$,
then we have the following property.
 \begin{theorem}\cite[Th.2.3]{ccs}
 Let $X$ be a Banach space. Then

 (i) $\mathscr C(X)$ is complete with respect to the Hausdorff metric, and

 (ii) $J: \mathscr C(X)\rightarrow\mathfrak S(\Omega)$ is a fully order-preserving  affine surjective isometry, i.e. for all $A, B\in\mathscr C(X)$
 \begin{equation}d_H(A,B)=\|J(A)-J(B)\|\equiv\sup_{x^*\in\Omega}|\sigma_A(x^*)-\sigma_B(x^*)|.\end{equation}
\end{theorem}

\begin{lemma}\cite[Lemma 2.5]{ccs}
Suppose that $X$ is a Banach space, and $\mathscr D=\mathscr D(X)$  is a  sub-semigroup of $\mathscr C=\mathscr C(X)$. Then $\mathscr D$ is closed if and only if  every $C\in\mathscr C$ satisfying the condition that for every $\varepsilon>0$ there exists $D\in\mathscr D(X)$ so that $C\subset D+\varepsilon B_X$ is in $\mathscr D$.
 \end{lemma}

All notions related to Banach lattices and abstract $M$ spaces are the same as in Lindenstrauss and Tzafriri \cite{lin}.  Recall that a partially ordered real Banach space $Z$ is called a Banach lattice provided

(i) $x\leq y$ implies $x+z\leq y+z$ for all $x,y,z\in Z$;

(ii) $ax\geq0,$ for all $x\geq0$ in $Z$ and $a\in\mathbb R^+$;

(iii) both $x\vee y$ and $x\wedge y$ exist for all $x,y\in Z$;

(iv) $\|x\|\leq\|y\|$ whenever $|x|\equiv x\vee{-x}\leq y\vee{-y}\equiv|y|$.

It follows from iv) and
\begin{equation}
|x-y|=|x\vee z-y\vee z|+|x\wedge z-y\wedge z|, \;\;{\rm for\;all\;}x,y,z\in Z,
\end{equation}
that the lattice operations are norm continuous.

By a sublattice of a Banach lattice $Z$ we mean a linear subspace $Y$ of $Z$ so that $x\vee y$ (and also $x\wedge y=x+y-x\vee y$) belongs to $Y$ whenever $x,y\in Y$. A lattice ideal $Y$ in $Z$ is a sublattice of $Z$ satisfying that $|z|\leq|y|$ for $z\in Z$ and for some $y\in Y$ implies $z\in Y$.

In the following discussion, unless stated explicitly otherwise, we always assume the lattice operations $\vee$ and $\wedge$ in a real-valued function space $C(K)$ are defined for $x,y\in C(K)$ and $k\in K$ by \begin{equation}
x(k)\vee y(k)=\max\{x(k), y(k)\}, \;\;x(k)\wedge  y(k)=\min\{x(k), y(k)\}.
\end{equation}
Given a fundamental sub-semigroup  $\mathscr D$ of $\mathscr C(X)$ (Definition 1.3), let $J$ be defined as (2.5), and put
 $E_\mathscr D=\overline{J\mathscr D-J\mathscr D}$.
\begin{theorem}\cite[Th. 3.1]{ccs}
Let $X$ be a Banach space. Then for any fundamental sub-semigroup  $\mathscr D$ of $\mathscr C(X)$, the space $E_\mathscr D$ is a Banach lattice.
\end{theorem}
%Recall $\mathscr K\equiv\mathscr K(X)$ denotes the (fundamental) subcone consisting of all nonempty compact convex sets of $X$.
\begin{theorem}\cite[Th. 3.2]{ccs}
Suppose that $X$ is a Banach space. Then

(i) $E_\mathscr K=\overline{J\mathscr K-J\mathscr K}$ is a lattice ideal of the Banach lattice $E_\mathscr C=\overline{J\mathscr C-J\mathscr C}$;

(ii) The quotient space $E_\mathscr C/E_\mathscr K$ is an abstract $M$ space. Therefore, there exist  a Banach function space $C(K)$  for some compact Hausdorff space $K$ and an order isometry $T$ from $E_\mathscr C/E_\mathscr K$  onto a sublattice of a $C(K)$.

(iii) $TQJ\mathscr C$ is a closed cone contained in the positive cone $C(K)^+$ of $C(K)$ , where $Q: E_\mathscr C\rightarrow E_\mathscr C/E_\mathscr K$ is the quotient mapping and $T:E_\mathscr C/E_\mathscr K\rightarrow C(K)$ is the corresponding order isometry.
\end{theorem}

Note that for a nonempty bounded subset $A\subset X^*$, $\|\cdot\|_A=\sigma_{A\cup-A}$ defines a continuous semi-norm on $X$.

\begin{theorem}\cite[Th. 5.5]{ccs} Let all the notions be the same as in Theorem 2.6.
 Then for every nonempty bounded set $F\subset C(K)^*$ of positive functionals satisfying that for each $0\neq u\in TQJ\mathscr C$, there exists $\varphi\in F$ so that $\langle\varphi, u\rangle>0$, the following formula defines  a homogeneous MNC $\mu$ on $X$.
 \begin{equation}
 \mu(B)=\|T[QJ\overline{{\rm co}}(B)]\|_F\;\;{\rm for\;all\;}B\in\mathscr B,
 \end{equation}
 where $\|u\|_F=\sup_{\varphi\in F\cup-F}\langle\varphi, u\rangle$\; for all $u\in C(K)$.

 In particular, the Hausdorff MNC $\beta$ on $X$ satisfies
 \begin{equation}
 \beta(B)=\|T[QJ\overline{{\rm co}}(B)]\|\;\;{\rm for\;all\;}B\in\mathscr B.
 \end{equation}
 %whenever we take $F$ to be the set $\{\delta_k: k\in K\}$ of all positive extreme points of $B_{C(K)^*}$.
 Therefore,  $TQJ\overline{\rm co}B=0$ if and only if $B$ is relatively compact.
\end{theorem}

 \begin{remark}Since the three mappings  $J: \mathscr C(X)\rightarrow C_b(\Omega)$, $Q: E_\mathscr{C}\rightarrow E_\mathscr{C}/E_\mathscr{K}$, and $T: E_\mathscr{C}/E_\mathscr{K}\rightarrow C(K)$ are order-preserving, $\mathbb T\equiv TQJ$ is an order preserving 1-Lipschitzian mapping from $\mathscr C(X)$ to $C(K)^+$, which is called the three-time order preserving affine mapping.
\end{remark}

\section{Representation of convex measures of noncompactness}

In this section, we are devoted to representation of convex MNCs.
We denote again by  $\mathscr B(X)$ the set of all nonempty bounded subsets of a Banach space $X$, and by $\mathscr C(X)$ (resp., $\mathscr K(X)$) the cone of all nonempty closed (resp., compact) convex subsets of  $X$ endowed the  operations addition $\oplus$ and scalar multiplication of sets, and endowed with the Hausdorff metric $d_H$. For a Banach space $X$, the Banach spaces $C_b(\Omega)$, $E_\mathscr C$, $E_\mathscr K$ and $C(K)$,  the mappings $J, Q, T$ and $\mathbb T=TQJ$ are the same as in Section 2. Let $V=\mathbb T\mathscr C(X)$. Then $V$ is a closed subcone of the positive cone $C(K)^+$ of $C(K)$.

\begin{lemma}
Let $X$ be a Banach space, and $f,g\in V\equiv\mathbb T\mathscr C(X)$. Then $f\leq g$ if and only if there exist $A, B\in\mathscr C(X)$ with $f=\mathbb TA$, $g=\mathbb TB$ such that for all $\eps>0$ there is $K_\eps\in\mathscr{K}(X)$ satisfying
\begin{equation}A\subset B+K_\eps+\eps B_X.\end{equation}
\end{lemma}
\begin{proof}
Sufficiency.  Suppose that $A, B\in\mathscr C(X)$ such that $f=\mathbb TA$, $g=\mathbb TB$.   Let $B_\eps=\overline{B+K_\eps+\eps B_X}\; (=B+K_\eps+\eps B_X$). Since $\mathbb T$ is affine and order-preserving with $\mathbb TC=0$ for all $C\in\mathscr K(X)$, we obtain that
\[\mathbb TA\leq\mathbb TB_\eps=\mathbb TB+\mathbb TK_\eps+\eps\mathbb TB_X.\]
Since $\eps$ is arbitrary, we obtain $f=\mathbb TA\leq\mathbb TB=g$.

Necessity. Note that $\mathbb T=TQJ$. Suppose $f\leq g\in V$. By the definition of $V$, there exist $A, B\in\mathscr C(X)$ such that \[TQJ(A)=\mathbb TA=f\leq g=\mathbb TB=TQJ(B).\]  Since $T: E_\mathscr C/E_\mathscr K\rightarrow T(E_\mathscr C/E_\mathscr K)\subset C(K)$ is fully order-preserving, we get
\[QJ(A)=(T^{-1}\mathbb T)A\leq(T^{-1}\mathbb T)B=QJ(B).\]
Therefore,
\[QJ(B)-QJ(A)\geq0.\]
By definition of the order-preserving quotient mapping $Q$, for every $\eps>0$, there exist $C, D\in\mathscr K(X)$ such that
\[J(B)-J(A)-(J(C)-J(D))\geq-\eps.\]
Positive homogeneity of $J(A),J(B), J(C)$ and $J(D)$ entails
\[J(B)-J(A)-(J(C)-J(D))\geq-\eps\|\cdot\|_{X^*}.\]
Note \[-(J(C)-J(D))=J(D)-J(C)\leq J(D-C).\] Then
\[J(B)+J(D-C)+\eps\|\cdot\|_{X^*}\geq J(A).\]
Let $K_\eps=D-C$. Then the inequality above is equivalent to
\[A\subset B+K_\eps+\eps B_X.\]
\end{proof}
\begin{corollary}
Let $X$ be a Banach space, and $f,g\in V\equiv\mathbb T\mathscr C(X)$. Then $f=g$ if and only if there exist $A, B\in\mathscr C(X)$ with $f=\mathbb TA$, $g=\mathbb TB$ such that for all $\eps>0$ there is $K_1, K_2\in\mathscr{K}(X)$ satisfying
\begin{equation}B\subset A+K_1+\eps B_X,\;{\rm and}\;\; A\subset B+K_2+\eps B_X.\end{equation}
\end{corollary}
\begin{lemma}
Let $f, g\in V$ with $\|f-g\|=r$. Then $|f-g|\leq r\mathbb T(B_X).$
\end{lemma}
\begin{proof} Let $f_C=\mathbb T C=TQJ(C)$ for all $C\in\mathscr{C}(X)$.
Since $f, g\in V$, there exist $A,B\in\mathscr C(X)$ such that \[f=\mathbb TA=f_A, \;g=\mathbb TB=f_B.\]
Since $T: Q(J(\mathscr{C}(X))\rightarrow V$ is a surjective order-preserving isometry,
\[T^{-1}(f_A)=QJ(A),\;T^{-1}(f_B)=QJ(B)\]
satisfy \[\|Q(J(A))-Q(J(B))\big\|=\|Q\Big(J(A)-J(B)\Big)\|=r.\]
By definition of the quotient mapping $Q: E_{\mathscr{C}}\rightarrow E_{\mathscr{C}}/E_{\mathscr{K}}$, for all $\eps>0$,
there exist  $K_1, K_2\subset\mathscr{K}X$ such that
\[\|\big(J(A)-J(B)\big)-\big(J({K_1})-J({K_2})\big)\|_{C_b(\Omega)}<r+\eps.\]
Positive homogeneity of the functions $J(A), J(B), J(K_1)$ and $J(K_2)$ entail
\[\big|\big(J(A)-J(B)\big)-\big(J({K_1})-J({K_2})\big)\big|<(r+\eps)\|\cdot\|_{X^*}.\]
Note $J(B_X)=\|\cdot\|_{X^*}.$ Then
\[J(A)+J({K_2})<(r+\eps)J(B_X)+J(B)+J({K_1}),\] and \[J(B)+J({K_1})<(r+\eps)J(B_X)+J(A)+J({K_2}).\]
Consequently,
\[\mathbb T(A)\leq \mathbb T(B)+(r+\eps)\mathbb T(B_X),\;\mathbb T(B)\leq \mathbb T(A)+(r+\eps)\mathbb T(B_X).\]
Therefore,
\[\big|\mathbb T(A)-\mathbb T(B)\big|\leq(r+\eps)\mathbb T(B_X).\]
Since $\eps$ is arbitrary, \[\big|f-g\big|=\big|f_A-f_B\big|\leq r\mathbb T(B_X).\]
\end{proof}

Recall (Definition 1.1 iv)) that for a Banach space $X$, $\mu: \mathscr B(X)\rightarrow \R^+$ is said to be a convex MNC provided it satisfies the following four properties.

(P1) [{\bf Noncompactness}] $B\in\mathscr B(X),\;\mu(B)=0 \Longleftrightarrow B$ is relatively compact;

(P2) [{\bf Monotonicity}] $A, B\in\mathscr B(X)\; {\rm with\;}A\supset B\;\Longrightarrow \mu(A)\geq\mu(B);$

(P3) [{\bf Convexification invariance}] $B\in\mathscr B(X)\Longrightarrow \mu({\rm co}{(B)})= \mu(B)$;

(P4) [{\bf Convexity}] $\mu(\lambda A+(1-\lambda)B)\leq\lambda\mu(A)+(1-\lambda)\mu(B),\;\forall A, B\in \mathscr B(X)\;{\rm and\;}0\leq\lambda\leq1.$ %Therefore, a regular measure, or, a sublinear measure of noncompactness is a convex one.

Every convex MNC admits the following basic properties.
\begin{theorem}
Let $X$ be a Banach space, and $\mu:\mathscr B(X)\rightarrow\mathbb R^+$ be a convex MNC. Then

i)  {\rm [{\bf Density determination}]} $\mu(\overline{B})=\mu(B),\;\forall B\in \mathscr B(X)$;

ii) {\rm[{\bf Translation invariance}]} $\mu(B+C)=\mu(B),\;\forall  B\in \mathscr B(X), C\in\mathscr K(X)$;

iii) {\rm[{\bf Negiligbility}]} $\mu(B\cup C)=\mu(B),\;\forall \; B\in \mathscr B(X), C\in\mathscr K(X)$;

iv) {\rm[{\bf Continuity}]} $\emptyset\neq B_{n+1}\subset B_n\in\mathscr B(X),\;n\in\N;\; \mu(B_n)\rightarrow0\Longrightarrow$\[\bigcap_n\overline{B_n}\neq\emptyset.\]
\end{theorem}
\begin{proof}
i) Monotonicity of $\mu$ entails
\[\mu(A)\leq\mu(\overline{A}),\;\;\forall A\in\mathscr{B}(X).\]
On the other hand, let $M=\||A\||\equiv\sup_{a\in A}\|a\|>0.$ Then
$\forall M\geq\eps>0$,
\[\overline{A}\subset A+\frac{\eps}{M} B_X\subset(1-\frac{\eps}M)A+\frac{\eps}M(M+1) B_X.\]
Convexity and monotonicity of $\mu$ imply
\[\mu(\overline{A})\leq (1-\frac{\eps}M)\mu(A)+\frac{\eps}M\mu((M+1) B_X).\]
Since $\eps$ is arbitrary, we obtain
\[\mu(\overline{A})\leq \mu(A)\leq\mu(\overline{A}).\]

ii) It is easy to check that for each $B\in\mathscr{B}(X)$, $f(t)\equiv\mu(tB), \;t\in\R$ defines a continuous convex function on $\R$. Indeed, since $\mu$ is a convex MNC,  $f$ is convex on $\R$. Consequently, $f$ is continous  on $\R$. $\forall\; B\in \mathscr B(X), C\in \mathscr K(X), 1\geq\lambda>0$,
\[\mu(B+C)=\mu[\lambda(\frac{1}\lambda C)+(1-\lambda)(\frac{1}{1-\lambda}B)]\]
\[\leq\lambda\mu(\frac{1}\lambda C)+(1-\lambda)\mu[(\frac{1}{1-\lambda}B)]\]
\[=(1-\lambda)\mu[(\frac{1}{1-\lambda}B)]\rightarrow\mu(B),\;\;\]
as $\lambda\rightarrow0^+$. Thus, \[\mu(B+C)\leq \mu(B),\;\forall x_0\in X,\;B\in\mathscr{B}(X).\]
We substitute $B+C$ for $B$ in the inequality above. Then
\[\mu(B)\leq\mu[-C+(B+C)]\leq\mu(B+C).\]
Therefore,
\[\mu(B+C)=\mu(B),\;\forall\;B\in\mathscr{B}(X),\;C\in\mathscr K(X).\]

iii) Given $B\in \mathscr B(X), x_0\in X$, choose any $b_0\in B$ and let $B_1=B-x_0$, $C=[0,x_0-b_0]\equiv\{\lambda(x_0-b_0): 0\leq\lambda\leq1\}.$ Then
by ii) we have just proven,
\[\mu(\{x_0\}\cup B)=\mu(\{0\}\cup B_1).\]
Note that $\{0\}\cup B_1\subset B_1+C$, and $C\in\mathscr K(X)$. Then again by ii),
\[\mu(\{0\}\cup B_1)\leq\mu(B_1+C)=\mu(B_1)=\mu(B).\]
Thus,
\[\mu(\{x_0\}\cup B)=\mu(B).\]
It follows that $\mu(B\cup F)=\mu(B)$ for any nonempty finite set $F\subset X$.
For any $C\in\mathscr{K}(X)$ and for any $1>\eps>0$, let $F_\eps$ be a finite set such that $C\subset F_\eps+\eps B_X$. Then
 \[\mu(B\cup C)\leq\mu\Big(B\cup(F_\eps+\eps B_X)\Big)=\mu\Big(B\cup\eps B_X\Big)\]
 \[\leq\mu\Big(\frac{1}{1+\eps}B+\frac{\eps}{1+\eps}(B\cup(1+\eps) B_X)\Big)\;\;\;\]
 \[\leq\frac{1}{1+\eps}\mu(B)+\frac{\eps}{1+\eps}\mu\Big(B\cup(1+\eps) B_X)\Big)\]
 \[\longrightarrow\mu(B),\;\;{\rm as}\;\;\eps\rightarrow0^+.\;\;\;\;\;\;\;\;\;\;\;\;\;\;\;\;\;\;\;\;\;\;\;\]
It follows that \[\mu(B\cup C)=\mu(B),\;\forall \; B\in \mathscr B(X), C\in\mathscr K(X).\]

iv) Let $\{B_n\}\subset\mathscr{B}(X)$ be a sequence of  nonempty subsets with $B_{n+1}\subset B_n\in\mathscr B(X),\;n\in\N;\; \mu(B_n)\rightarrow0$.
For each $n\in\N$, choose any $x_n\in B_n$. Then by iii),
\[\mu(\{x_n\})=\mu(\{x_n\}_{n\geq k})\leq\mu(B_k)\rightarrow0,\;{\rm as}\;k\rightarrow\infty.\]
It follows that $\{x_n\}$ is relatively compact. Let $x_0$ be a cluster point of $\{x_n\}$. Then $x_0\in\bigcap_n\overline{B_n}.$
\end{proof}
For $B\in\mathscr B(X)$, we will simply denote $\mathbb T(B)$ by $f_B$ in the sequel.
\begin{lemma}
Let $\mu$ be a convex MNC on a Banach space $X$. Then

i) $\digamma(f_B)=\mu(B),\,B\in\mathscr B(X)$ defines a monotone convex function on $V$;

ii) For each $r>0$, $\digamma$ is bounded by $b_r\equiv\mu(rB_X)$ on  $V\cap(rB_{C(K)})$;

iii) For each $f\in V$,
\begin{equation}p(g)=\lim_{t\rightarrow 0^+}\frac{\digamma(f+tg)-\digamma(f)}t, \;g\in V\end{equation}
defines a non-negative sublinear functional $p$ on $V$ satisfying
\begin{equation}
p(g)\leq\digamma(f+g)-\digamma(f),\;\;\forall g\in V.
\end{equation}

%iv) $p$ is $c$-Lipschitz on $V$, where $c=\digamma(f_{B_{X}})-\digamma(f)+p(f).$
\end{lemma}

\begin{proof}
i)  We first show that $\digamma$ is well-defined on $V$. Indeed, it follows from Corollary 3.2, for all $A, B\in\mathscr B(X)$,
$f_A=f_B$ if and only if $\forall\;\eps>0, \;\exists \;K_1, K_2\in\mathscr{K}(X)$ such that
\[B\subset K_1+A+\eps B_X,\;{\rm and\;}\;A\subset K_2+B+\eps B_X.\]
Thus, \[\digamma(f_A)=\digamma(f_B)\Longleftrightarrow \mu(A)=\mu(B).\]

 Let $f_A, f_B\in V$ and $\lambda\in[0,1]$. Then
\[\digamma[\lambda f_A+(1-\lambda)f_B]=\digamma[\mathbb T(\lambda A+(1-\lambda)B)]\]
\[=\mu(\lambda A+(1-\lambda)B)\leq \lambda\mu(A)+(1-\lambda)\mu(B)\]
\[=\lambda\digamma(f_A)+(1-\lambda)\digamma(f_B).\]
Therefore, the convexity of $\digamma$ is shown.

To show the monotonicity of $\digamma$, let $f_A\geq f_B\in V$. Then it follows from Lemma 3.1 that  for all $\eps>0$ there is $K_\eps\in\mathscr{K}$ satisfying
\[B\subset A+K_\eps+\eps B_X.\]
It follows from Theorem 3.4 ii) that
\[\digamma(f_B)=\mu(B)\leq\mu(A+\eps B_X)\leq(1-\eps)\mu(\frac{1}{1-\eps}A)+\eps\mu(B_X)\]
\[=(1-\eps)\digamma(\frac{1}{1-\eps}f_A)+\eps\digamma(f_{B_X}).\]
Note that $g: \R^+\rightarrow\R^+$ defined for $t\in\R^+$ by $g(t)=\digamma(tf_A)$ is a convex function. It is continuous at $t=1$. Thus, let $\eps\rightarrow0^+$  in the inequalities above. Then
\[\digamma(f_B)\leq\digamma(f_A).\]

ii) Let $r>0$, $f=f_B\in V\cap(rB_{C(K)})$. Then $\|f\|\leq r$. By Lemma 3.3,  $f\leq rf_{B_X}=f_{rB_X}$. Due to the monotonicity of $\digamma$,
\[0\leq\digamma(f_B)\leq \digamma(f_{rB_X})=\mu(rB_X)=b_r.\]

iii) Let $f,g\in V$.  Since $V$ is a cone, $\digamma(f+tg)$ is clearly well-defined for all $t\geq0$. By i) we have just proven, $\digamma(f+tg)$ is convex in $t\in\R^+$. For all  $0<t<s\leq 1$,
\begin{equation}\nonumber
 \begin{aligned}
0\leq G(t)&\equiv\frac{\digamma(f+tg)-\digamma(f)}{t}\\
&=\frac{\digamma[\frac{t}{s}(f+sg)+\frac{s-t}{s}f]-\digamma(f)}{t}\\
&\leq\frac{\frac{t}{s}\digamma(f+sg)+\frac{s-t}{s}\digamma(f)-\digamma(f)}{t}\\
&=\frac{\frac{t}{s}\digamma(f+sg)-\frac{t}{s}\digamma(f)}{t}\\
&=\frac{\digamma(f+sg)-\digamma(f)}{s}\\
&=G(s).
\end{aligned}
\end{equation}
Therefore, $G(t)$ is increasing  $t\in\R^+$. Consequently,
\[0\leq p(g)=\lim\limits_{t\rightarrow 0^+}G(t)\leq G(1)\leq \digamma(f+g)-\digamma(f).\]
Since $\digamma$ is convex on $V$,  $p$ is a non-negative sublinear functional on $V$ and satisfies
\[p(g)\leq \digamma(f+g)-\digamma(f), \;\forall g\in V.\]
\end{proof}

\begin{lemma}
Let $\mu$ be a convex measure of noncompactness on a Banach space $X$, and   $\digamma$ be defined as Lemma 3.5. Then $\digamma$ is continuous on $V$.
\end{lemma}

\begin{proof}
We first show that $\digamma$ is continuous at $0$. Let $\{u_n\}\subset V$ be a null sequence with  $r_n=\|u_n\|,\;n\in\N$. By Lemma 3.3, for all sufficiently large $n\in\N$, \[0\leq\digamma(u_n)\leq\digamma(\|u_n\|f_{B_X})\leq\|u_n\|\digamma(f_{B_X})\rightarrow0.\]
Therefore, $\digamma$ is continuous at $0$.

To show that $\digamma$ is continuous at each point $v_0\in V$, let $\{v_n\}\subset V$ such that $v_n\rightarrow v_0$, and let $r_n=\|v_n-v_0\|$. Again by Lemma 3.3,
\[|v_n-v_0|\leq r_nf_{B_X},\;\;n=0,1,\cdots.\]
This implies that
\begin{equation}
v_n\leq v_0+r_nf_{B_X}, \;v_0\leq v_n+r_nf_{B_X},\;n=1,2,\cdots.
\end{equation}
It follows from the first inequality of (3.5) that
\[v_n\leq\frac{1}{1+r_n}[(1+r_n)v_0]+\frac{r_n}{1+r_n}[(1+r_n)f_{B_X}].\]
Since $\digamma$ is convex,
\begin{equation}
\digamma(v_n)\leq\frac{1}{1+r_n}\digamma((1+r_n)v_0)+\frac{r_n}{1+r_n}\digamma((1+r_n)f_{B_X}).
\end{equation}
Since $\xi$ (defined for $t\in\R^+$ by $\xi(t)\equiv\digamma((1+t)v_0)$) is a  convex function, it is necessarily continuous in $t\in\R^+$. Therefore,
\begin{equation}
\frac{1}{1+r_n}\digamma((1+r_n)v_0)\rightarrow\digamma(v_0),\;{\rm and\;}\frac{r_n}{1+r_n}\digamma((1+r_n)f_{B_X})\rightarrow0,
\end{equation}
 (3.6) and (3.7) together imply
\begin{equation}\limsup_n\digamma(v_n)\leq\digamma(v_0).
\end{equation}
%On the other hand, let $u_n=v_n/(1+r_n)$.
It follows from the second inequality of (3.5) that
\[\frac{1}{1+r_n}v_0\leq\frac{1}{1+r_n}v_n+\frac{r_n}{1+r_n}f_{B_X}.\]
Convexity of $\digamma$ implies that
\begin{equation}
\digamma(\frac{1}{1+r_n}v_0)\leq\frac{1}{1+r_n}\digamma(v_n)+\frac{r_n}{1+r_n}\digamma(f_{B_X}).
\end{equation}
Therefore,
\begin{equation}\liminf_n\digamma(v_n)\geq\digamma(v_0).
\end{equation}
Continuity of $\digamma$ at $v_0$ follows from (3.8) and (3.10).
\end{proof}
For a subcone $C$ of a Banach space $Z$, we say a linear functional $z^*\in Z^*$ is a positive functional on $C$ if it is non-negative valued on $C$, or, equivalently, $z^*|_{X_C}$ is a positive functional of the subspace $X_C\equiv C-C$ with respect to the ``positive" cone $C$ in the usual sense.
\begin{lemma}
Let $V_0\subset V$ be a subcone of $V$ with $f_{B_X}\in V_0$, and let $E_{V_0}=V_0-V_0$. Then for every functional $x^*\in E^*_{V_0}$ which is positive on $V_0$, we have
\begin{equation}\|x^*\|_{E_{V_0}}=\langle x^*,f_{B_X}\rangle.\end{equation}
\end{lemma}
\begin{proof}
Let $u,v\in V_0$, $h=u-v\in E_{V_0}$ with $\|h\|\leq1$. Then by Lemma 3.3,
\[u\leq v+f_{B_X},\;\;{\rm and\;\;}v\leq u+f_{B_X}.\]
Since $x^*$  is positive on $V_0$, we obtain
\[\langle x^*,u\rangle\leq\langle x^*,v\rangle+\langle x^*,f_{B_X}\rangle\]
and
\[\langle x^*,v\rangle\leq\langle x^*,u\rangle+\langle x^*,f_{B_X}\rangle.\]
Therefore,
\[|\langle x^*,h\rangle|=|\langle x^*,u-v\rangle|\leq\langle x^*,f_{B_X}\rangle.\]
The inequality above is clearly equivalent to (3.11).
\end{proof}
\begin{lemma}
Suppose that  $V_0\subset V$ is a subcone of $V$, and $u\in V_0$ is in the relative interior ${\rm int}(V_0)$ of $V_0$. Let the function $\digamma$ be defined on $V$ as Lemma 3.5.  If $u$ is a G\^{a}teaux differentiability point of $\digamma |_{V_0}$ (the restriction of $\digamma$ to $V_0$). Then its relative G\^{a}teaux derivative $x^*=d_G\digamma |_{V_0}(u)$ is a positive functional on $V_0$.
\end{lemma}
\begin{proof}
This is an immediate consequence of  definition of G\^{a}teaux derivative and the monotonicity of $\digamma$ on $V$.
\end{proof}

\begin{lemma}
Suppose that $g$ is a continuous convex function on a closed convex subset $D$ of a Banach space $X$ with ${\rm int}(D)\neq\emptyset$. Then the following mean-value theorem holds.

i) $\forall x,y\in {\rm int}(D)$, there exist $\xi\in[x,y]\equiv\{\lambda x+(1-\lambda)y:\lambda\in[0,1]\}$ and $x^*_\xi\in\partial g(\xi)$ such that
\begin{equation}g(y)-g(x)=\langle x^*_\xi,y-x\rangle.\end{equation}

ii)  $\forall x^*\in\partial g(x),\;y^*\in\partial g(y)$, we have
\begin{equation}\langle y^*,y-x\rangle\geq g(y)-g(x)\geq\langle x^*,y-x\rangle.\end{equation}
\end{lemma}
\begin{proof} Since  $g$ is a continuous convex function on  ${\rm int}(D)$, $\partial g: {\rm int}(D)\rightarrow 2^{X^*}$  is nonempty $w^*$-compact-valued and norm-to-$w^*$ upper semicontinuous at each point of ${\rm int}(D);$ and $\partial g$ is maximal monotone on ${\rm int}(D)$. (See, for instance, R.R. Phelps \cite{ph}.)

i) For each fixed $z\in[x,y]$, since $\partial g(z)$ is nonempty $w^*$-compact and convex, $\{\langle z^*,y-x\rangle: z^*\in\partial g(z)\}$ is an interval, say, $[a_z,b_z]\subset\R,$
where $a_z=\min\{\langle z^*,y-x\rangle: z^*\in\partial g(z)\}$, and $b_z=\max\{\langle z^*,y-x\rangle: z^*\in\partial g(z)\}.$ Since $\partial g$ is monotone, for $0\leq\alpha<\beta\leq1$, $z_1=\alpha x+(1-\alpha)y$,  $z_2=\beta x+(1-\beta)y$, and for all $z_1^*\in\partial g(z_1)$ and $z_2^*\in\partial g(z_2)$, we have
\[\langle z_1^*,y-x\rangle\geq\langle z_2^*,y-x\rangle.\]
Since  $\partial g: {\rm int}(D)\rightarrow 2^{X^*}$  is norm-to-$w^*$ upper semicontinuous, for every selection $\varphi$ of $\partial g$ on $[x,y]$ and for every $w^*$-cluster point $z^*_0$ of $\{\varphi(\beta x+(1-\beta)y): \beta\rightarrow \alpha^+\}$ we have $z_0^*\in\partial g(z_1)$. Thus, \[\bigcup\{[a_z,b_z]: z\in[x,y]\}\] is again an interval. Consequently, (3.12) holds.

ii) It follows from definition of the subdifferential mapping $\partial g$.

\end{proof}

\begin{lemma} Suppose that $X$ is a Banach space, $\mu$ is a convex measure of noncompactness on $X$, and that $V_0\subset V$ is a closed subcone of $V$ with $f_{B_X}\in V_0$ such that $E_{V_0}=V_0-V_0$ is a finite dimensional subspace. Let the function $\digamma$ on $V$  be defined as Lemma 3.5. Then

 i) for each $r>0$, $\digamma$ is $c_r$-Lipschitzian on $V_{0,r}\equiv V_0\cap rB_{C(K)}$, where $c_r=\digamma((1+r)f_{B_X})=\mu\big((1+r)B_X\big)$;

 ii) if, in addition, $\mu$ is a homogeneous measure of noncompactness, then $c_r=\digamma(f_{B_X})=\mu(B_X).$
\end{lemma}
\begin{proof}
i) Since $V_0$ is a finite dimensional subcone of $V$, the relative interior ${\rm int}_{rel}(V_{0,r})$ is not empty. By Lemma 3.6,   $\digamma$ is convex and continuous  on $V$. It follows that the restriction $\digamma|_{V_0}$ is locally Lipschitzian on ${\rm int}_{rel}(V_0)$.  Therefore, $\digamma|_{V_0}$ is G\^ateaux differentiable at each point of a dense $G_\delta$-subset $M$ of  ${\rm int}_{rel}(V_0)$. Due to Lemma 3.8, for each $u\in M$, $\varphi\equiv d_G\digamma|_{V_0}(u)$ is a positive functional on $V_0$ and with $\|\varphi\|_{E_{V_0}}=\langle\varphi, f_{B_{X}}\rangle$. On the other hand, by Lemma 3.5,
\begin{equation}\langle\varphi, f_{B_{X}}\rangle\leq\digamma(u+f_{B_X})-\digamma(u)\leq\digamma(u+f_{B_X}).\end{equation} Since $u\in V_{0,r}$, by Lemma 3.3 we obtain that $u\leq f_{rB_X}=rf_{B_X}$.  Consequently,
\begin{equation}\|\varphi\|_{E_{V_0}}=\langle\varphi, f_{B_{X}}\rangle\leq\digamma((1+r)f_{B_X})).\end{equation}
Let $N=\{d_G\digamma|_{V_0}(u): u\in M\}$. Since the subdifferential mapping $\partial\digamma|_{V_0}: {\rm int}_{rel}(V_{0,r})\rightarrow 2^{E_{V_0}^*}$ is norm-to-norm upper semi-continuous, the norm closure $\overline{N}$ in $E_{V_0}^*$ contains a selection of $\partial\digamma|_{V_0}$ on $V_{0,r}$. Thus, for all $u,v\in V_{0,r}$, there exist $\varphi_1\in\partial\digamma|_{V_0}(u)\cap\overline{N}$ and $\varphi_2\in\partial\digamma|_{V_0}(v)\cap\overline{N}$. It follows from Lemma 3.9,
\begin{equation}
\langle\varphi_1, v-u\rangle\leq\digamma(v)-\digamma(u)\leq\langle\varphi_2, v-u\rangle.
\end{equation}
Note that $\|\varphi\|_{E_{V_0}}\leq\digamma((1+r)f_{B_X}))$ for all $\varphi\in\overline{N}$. Then it follows from (3.16) that
\[|\digamma(v)-\digamma(u)|\leq c_r\|v-u\|,\] this says that $\digamma$ is $c_r$-Lipschitzian on  ${\rm int}_{rel}(V_{0,r})$.

ii) If, in addition, $\mu$ is a homogeneous measure of noncompactness, then $\digamma$ is a non-negative sublinear functional on $V$. It follows from (3.14) that $c_r=\digamma(f_{B_X})$.
\end{proof}

Now,  we are ready to state and prove the main result of this section, namely, the representation theorem of convex MNC as follows.

\begin{theorem} Suppose that $X$ is a Banach space. Then there is a Banach space $C(K)$ for some compact Hausdorff space $K$ such that for every convex MNC $\mu$ on $X$, there is a function $\digamma$ on the cone $V\equiv\mathbb T(\mathscr{B}(X))\subset C(K)^+$ satisfying

i) $\mu(B)=\digamma(\mathbb TB)$, for all $B\in\mathscr{B}(X)$;

ii)  $\digamma$ is nonnegative-valued  convex and monotone on $V$;

iii) $\digamma$ is bounded by $b_r=\digamma(r\mathbb TB_X)$ on $V\bigcap(rB_{C(K)})$, for all $r\geq0$;

iv) $\digamma$ is $c_r$-Lipschitian on $V\bigcap(rB_{C(K)})$, for all $r\geq0$, where $c_r=\digamma\big((1+r)\mathbb TB_X\big)=\mu\big((1+r)B_X\big)$;

v) In particular, if $\mu$ is a sublinear MNC, then we can  take $c_r=\mu(B_X)$ in iv).
\end{theorem}
\begin{proof}
Given a convex measure of noncompactness $\mu$ on $X$, let
\[\digamma(\mathbb TB)=\mu(B),\;\;\forall B\in\mathscr{B}(X).\]
Then by Lemma 3.5, i), ii) and iii) follow. Next, we will show iv).
Given $r\geq0$, $f,g\in V$ with $\|f\|,\;\|g\|\leq r$, let $V_0\subset V$ be the subcone generated by $\{f,g, \mathbb T(B_X)$. By Lemma 3.10 i), $\digamma$ is $c_r$-Lipschitian on $V_0\bigcap{rB_{C(K)}}$. Therefore,
\[\big|\digamma(f)-\digamma(g)\big|\leq c_r\|f-g\|.\]
Consequently, iv) follows.

v) follows from Lemma 3.10 ii).
\end{proof}
The following result  is a converse version of Theorem 3.11 for sublinear MNCs.
\begin{theorem}
Suppose that $\mu$ is a sublinear MNC defined on a Banach space $X$, and that $C(K)$ is the Banach function space with respect to $X$ defined as in Theorem 3.11. Then there is a bounded subset $M\subset {C(K)^*}^+$ of positive functionals such that
\[\mu(B)=\sup_{\varphi\in M}\langle\varphi, \mathbb T(B)\rangle,\;\;\forall B\in\mathscr{B}(X).\]
\end{theorem}
\begin{proof}
Since $\mu$ is a sublinear MNC, by Theorem 3.11,  the corresponding convex function $\digamma$ is $c (=\mu(B_X)$)-Lipschitzian on the positive cone $V=\mathbb T(\mathscr{B}(X))$. Let
\[\gimel(f)=\inf\{\digamma(g)+c\|f-g\|_{C(K)}: g\in V\},\;\;\forall f\in C(K).\]
Then by a standard argument we can see that $\gimel$ is c-Lipschizian on the whole space $C(K)$ and with $\gimel|_V=\digamma$ on $V$. Let $M^\prime=\partial\gimel(V)\equiv\bigcup_{v\in V} \partial\gimel(v)$. Then $\digamma(v)=\sup_{\varphi\in M^\prime}\langle\varphi,v\rangle$. Next, we will show that each $\varphi\in M^\prime$ is a positive functional on $V$.
Let $F\subset C(K)$ be a finite dimensional subspace so that $F=V_F-V_F$, where $V_F\equiv V\bigcap F$. Then
$\partial\gimel(v)|_{F}\subset\partial\gimel|_{F}(v).$
Since $F$ is finite dimensional, $\gimel|_{F}$ is densely (G\^{a}teaux) differentiable in $F$, hence, densely differentiable in $V_F$. Let $D$ be the dense subset of $V_F$ such that at each point of $D$ $\gimel_{F}$ is differentiable. By Lemma 3.8, for each  point $v\in D$ of  $\gimel|_{F}$, the G\^{a}teaux derivative  $d_G\gimel|_{F}(v)$ is a positive functional on the space $F$. Thus, all elements in $\overline{\rm co}(d_G\gimel|_{F}(v): v\in D)$ are positive functionals on the space $F$. Note $\partial\gimel(V)|_{F}\subset\overline{\rm co}(d_G\gimel|_{F}(v): v\in D).$ Then we see that $\partial\gimel(V)|_F$ is consisting of positive functionals on $F$. Since $F$ is arbitrary, each element in $\partial\gimel(V)$ is a positive functional on the subspace $V-V$ of $C(K)$.

For each $\varphi\in M^{\prime}$, it has a unique  decomposition $\varphi=\varphi^+-\varphi^-$ and with $\|\varphi\|=\|\varphi^+\|+\|\varphi^-\|$, where $\varphi^\pm\in {C^*(K)}^+$, the cone of all  positive functionals of $C(K)$. Since $\varphi|_{V-V}$ is a positive functional of $V-V$, $\varphi|_{V-V}=\varphi^+|_{V-V}$. Thus, we finish the proof by letting $M=\{\varphi^+: \varphi\in M^\prime\}$.
\end{proof}

\section{Applications to a class of basic integral inequalities}
In this section, we will use the representation theorem of convex MNC (Theorem 3.11)  to establish the basic integral inequalities (1.8) and (1.10) with respect to the initial problem (1.5) presented in the first section, which can be understood as an extension, improvement and unification of some classical results related to the problem.

The letter $X$ still denotes a real Banach space, and $I$ is the interval $[0,a]\subset\R^+$ with $a>0$. We use $(I,\Sigma,m)$ to denote the Lebesgue measure space. For a set $A$, $\chi_A$ stands for the characteristic function of $A$.  Unless stated otherwise, all notions and symbols will be the same as previously defined. We use $L_p(I, X)$ ($0< p\leq\infty$) to denote the space of all $X$-valued measurable functions defined on $I$ such that $|f|^p$ is Lebesgue-Bochner integrable. $L_0(I, X)$ stands for the space of all $X$-valued strongly measurable functions.

We are going to convert whether the integral equality is true to whether the mapping $\mathbb J_G: I\rightarrow C_b(\Omega)$ defined blow  is (Lebesgue-Bochner) measurable, where $\Omega=B_{X^*}$ and $C_b(\Omega)$ denotes again the Banach space of all continuous bounded functions on $\Omega$ endowed with the sup-norm. For a Banach space $X$, the Banach spaces $E_\mathscr C$, $E_\mathscr K$ and $C(K)$,  the mappings $J, Q, T$ and $\mathbb T=TQJ$, the positive cone $V=\mathbb T\mathscr C$ are the same as in Section 2.

\begin{definition}
Let $G\subset L_0(I,X)$ be a nonempty subset satisfying $\sup\|G(t)\|\equiv\sup_{g\in G}\|g(t)\|<\infty$ a.e. The mapping $\mathbb J_G: I\rightarrow C_b(\Omega)$ is defined for $t\in I$ by
\begin{equation}
\mathbb J_G(t)(x^*)=\sup_{g\in G}\langle x^*,g(t)\rangle,\;x^*\in\Omega.
\end{equation}

Note that $\mathbb J_G(t)=J(G(t))$, where $G(t)=\{g(t): g\in G\}$.
If it arises no confusion, we simply write $\mathbb J$ for $\mathbb J_G.$

\end{definition}

%extending results of Goebel and Rzymowski\cite{goe}, Kunze and Schl\"{u}chtermann\cite{kun}, etc.

%We assume that $X$ is a Banach space, $I=[0,a]\subset \R$, $\mathbb{R}_+=[0,+\infty)$. %The word measurable is always referred to the Lebesgue measure on the real line .
We recall some notions and basic properties related to  $X$-valued functions defined on the interval $I$ (from Definition 4.2 to Lemma 4.6), which can be found in J. Diestel \cite{di}.
\begin{definition}
 Let $f: I\rightarrow X$ be a function.

 i) $f$ is said to be a simple function provided there exist a finite $\Sigma$-partition $\{E_j\}_{j=1}^n$ and $n$ vectors $\{x_j\}_{j=1}^n\subset X$ such that $f=\sum_{j=1}^{n} x_j\chi _{E_j}$. The integral of $f$ is defined by $\int_Ifdm=\sum_jx_jm(E_j).$

 ii) $f$ is called (strongly) measurable if there is a sequence $\{f_n\}$ of simple functions such that $f_n(s)\rightarrow f(s)$ for almost all $s\in I$. If, in addition,
 the limit $\lim_n\int_If_ndm$ exists, then we say that $f$ is (Lebesgue-Bochner) integrable and $\int_Ifdm\equiv\lim_n\int_If_ndm$ is called the integral of $f$ on $I$.

 iii) We say that $f$ is weakly measurable if for each $x^*\in X^*$ the numerical function $\langle x^*,f\rangle$ is ($\Sigma$-) measurable.

\end{definition}

%\begin{lemma}
%Let $X,Y$ be Banach spaces, $f: I\rightarrow X$ be strongly measurable and $g: X\rightarrow Y$ be continuous, then $g\circ f:I\rightarrow Y$  is strongly measurable.
%\end{lemma}

%\begin{proof}
 %For strongly measurable $f: I\rightarrow X$, there exist a sequence of simple functions $f_n:I\rightarrow X\;(n\geq 1)$, such that $\lim\limits_{n\rightarrow\infty} f_n(t)=f(t)\;\; a.e.\;t\in I$. Then $g\circ f_n:I\rightarrow Y$ are still simple functions, and
 %$$ \lim_{n\rightarrow\infty} (g\circ f_n)(t)= g(\lim_{n\rightarrow\infty} f_n(t))=g(f(t)),\;\;a.e.\;t\in I,$$
% that is, $g\circ f:I\rightarrow Y$  is strongly measurable.
 %\end{proof}

\begin{lemma}
A function $f:I\rightarrow X$ is (Lebesgue-Bochner) integrable if and only if $f$ is strongly measureable and $\int_I \|f\|dm<\infty.$
\end{lemma}

\begin{lemma}{\rm(Pettis's measurability theorem)}\; A function $f:I\rightarrow X$ is strongly measurable if and only if

(1) $f(I)\subset X$ is essentially separable, i.e., there exists a null set $I_0\subset I$  such that $f(I\backslash I_0)$ is  separable; and

(2) $f$ is weakly measurable.
\end{lemma}

\begin{lemma} %{\rm(\cite{di})}
Let $F$ be a closed linear operator defined inside $X$ and having values in a Banach space $Y$. If both $f:I\rightarrow X$ and $Ff$ are Bochner integrable, then $$F\Big(\int_I fds\Big)=\int_I Ff ds.$$
\end{lemma}

\begin{lemma}[Jensen's inequality]  Assume $f:[0,1]\rightarrow X$ is integrable, and $p: X\rightarrow\mathbb{R}$ is a continuous convex function. If $p\circ f$ is integrable, then
\[p\;\Big(\int_0^1fds\Big)\leq \int_0^1\Big(p\circ f\Big)ds.\]
\end{lemma}

\begin{lemma}
Let $X$ be a Banach space, $I=[0,a]$, $G\subseteq L_1(I,X)$ be a nonempty set,  $\psi\in L_1(I,\mathbb{R}_+)$ such that $\sup_{d\in G}\|g(t)\|\leq \psi(t)\;\;a.e.\;t\in I$ . Assume that $\mathbb J_G: I\rightarrow C_b(\Omega)$ is strongly measurable. then $TQ\mathbb J_G: I\rightarrow C(K)$ defined for $t\in I$ by $TQ\mathbb J_G(t)=TQJ\Big(G(t)\Big)$ is integrable on $I$ and satisfies
\begin{equation}0\leq TQJ\Big(\int_0^tG(s)ds\Big)\leq \int_0^t TQJ\Big(G(s)\Big)ds.\end{equation}
\end{lemma}

\begin{proof} Since $\mathbb J_G(t)=J\Big(G(t)\Big)$ is strongly measurable with respect to $t$, and since  $\psi\in L_1(I,\mathbb{R}_+)$, it follows from
  \begin{equation}\nonumber
 \begin{aligned}
  \|J\Big(G(t)\Big)\|_{C_b(\Omega)}&=\sup_{x^*\in \Omega}\sup_{g\in G}|\langle x^*,g(t)\rangle|\\
  &\leq \sup_{x^*\in \Omega}\sup_{g\in G}\|x^*\|\cdot \|g(t)\|\\
  &\leq\psi(t),
    \end{aligned}
\end{equation}
that $\|J(G(t))\|_{C_b(\Omega)}\in L_1(I,\mathbb{R}^+)$. By Lemma 4.3, $J(G(t))$  is integrable on $I$. Consequently, $TQJ(G(t))$  is integrable.

Note that every $x^*\in \Omega=B_{X^*}$,  $x^*$ can be regarded as the evaluation function $\delta_{x^*}\subset C^*_b(\Omega)$. By Lemma 4.5,
\begin{equation}\nonumber
 \begin{aligned}
J\Big(\int_0^t G(s)ds\Big)(x^*)&=\sup_{g\in G}\langle x^*,\int_0^t g(s)ds\rangle\\
&=\sup_{g\in G}\int_0^t \langle x^*,g(s)\rangle ds\\
&\leq\int_0^t \sup_{g\in G}\langle x^*,g(s)\rangle ds\\
&=\int_0^t J\Big(G(s)\Big)(x^*)ds\\
&=\Big(\int_0^t J(G(s))ds\Big)(x^*).
\end{aligned}
\end{equation}
Thus, in the order induced by the positive cone of $C_b(\Omega)$, we have
\begin{equation}\nonumber
J\Big(\int_0^tG(s)ds\Big)\leq \int_0^t J\Big(G(s)\Big)ds.
\end{equation}
Since the quotient mapping $Q$ is order preserving, it follows from Lemma 4.5
\begin{equation}\nonumber
 \begin{aligned}
QJ\Big(\int_0^tG(s)ds\Big)&\leq Q\int_0^t J\Big(G(s)\Big)ds\\
&=\int_0^t QJ\Big(G(s)\Big)ds.
\end{aligned}
\end{equation}
Since $T$ is an order isometry,
\begin{equation}\nonumber
 \begin{aligned}
TQJ\Big(\int_0^tG(s)ds\Big)&\leq T\int_0^tQJ\Big(G(s)\Big)ds\\
&=\int_0^t TQJ\Big(G(s)\Big)ds.
\end{aligned}
\end{equation}
Since $TQJ\Big(\mathscr C(X)\Big)\subset C(K)^+$ (Theorem 2.6 (iii)), and since $TQJ\Big(\mathscr C(X)\Big)=TQJ\Big(\mathscr B(X)\Big)$,
\begin{equation}\nonumber
0\leq TQJ(\int_0^tG(s)ds\Big)\leq \int_0^t TQJ\Big(G(s)\Big)ds.
\end{equation}
\end{proof}

%\begin{lemma}(\cite{nu})
%Let $\mu$ be a homogeneous measure of noncompactness on a Banach space $X$, then there exists a constant $c\geq 0$ such that
%$$\mu(B)\leq c\beta(B),\;{\rm for\; all}\; B\in \mathscr B(X). $$
%\end{lemma}

%\begin{lemma}(\cite{cheng1})
%The Hausdorff measure of noncompactness $\beta$ on a Banach space $X$ can be reformulated as follows
%$$\beta(B)=\|TQJ(B)\|_{C(K)},\;{\rm for\; all}\; B\in \mathscr B(X). $$
%\end{lemma}

\begin{theorem}
Let $\mu$ be a convex MNC on a Banach space $X$,  $I=[0,a]$, $G\subseteq L_1(I,X)$ be a nonempty bounded subset, and  $\psi\in L_1(I,\mathbb{R}^+)$ such that $\sup_{g\in G}\|g(t)\|\leq \psi(t)\;\;a.e.\;t\in I$. Assume that $\mathbb J_G: I\rightarrow C_b(\Omega)$ is strongly measurable. Then $\mu\Big(G(t)\Big)$ is measurable and
\begin{equation}
\mu\Big(\int_0^t G(s)ds\Big)\leq \frac{1}t\int_0^t \mu\Big(tG(s)\Big)ds,\;\;\forall 0<t\leq a.
\end{equation}
\end{theorem}

\begin{proof} By Lemma 4.7, $\mathbb T(G)=TQJ(G): I\rightarrow V\subset C(K)^+$ defined for $t\in I$ by $\mathbb T(G(t))=TQJ(G(t))=TQ\mathbb J_G(t)$ is integrable on $I$. On the other hand, it follows from Theorem 3.11 that there is a nonnegative-valued continuous convex function $\digamma$ on $V\equiv TQJ\Big(\mathscr{C}(X)\Big)=\mathbb T\Big(\mathscr{B}(X)\Big)\subset C(K)^+$, which is Lipschitzian on each bounded subset of $V$ such that
\[\digamma\Big(\mathbb T(B)\Big)=\mu(B),\;\;\forall B\in\mathscr{B}(X).\]
Therefore,
\begin{equation}\digamma\Big(\mathbb T(G)\Big)=\digamma\Big(TQJ(G)\Big)=\mu(G): I\rightarrow\mathbb R^+\end{equation}
is measurable on $I$. Note that $\mathbb J_{\lambda G}=\lambda\mathbb J_G$ for all fixed $\lambda>0$. $\mathbb J_{tG}: I\rightarrow C_b(\Omega)$ is again strongly measurable for all fixed $0<t\leq a$.
Consequently, for all fixed $0<t\leq a$, \[\digamma\Big(\mathbb T(tG)\Big)=\digamma\Big(tTQJ(G)\Big)=\mu(tG): I\rightarrow\mathbb R^+\]
is again measurable on $I$.
Now, given $0<t\leq a$, if $\int_0^t \mu\Big(tG(s)\Big)ds=\infty$, then we finish the proof. If $\int_0^t \mu\Big(tG(s)\Big)ds<\infty$, then $\mu\Big(tG(s)\Big)=\digamma\Big(t\mathbb T(G(s))\Big)$ is integrable with respect to $s\in[0,t]$, and with
\begin{equation}\frac{1}t\int_0^t \mu\Big(tG(s)\Big)ds=\int_0^1 \mu\Big(tG(ts)\Big)ds.\end{equation}

Since $\psi\in L_1(I,\mathbb{R}^+)$ with $\sup_{g\in G}\|g(s)\|\leq \psi(s)$ a.e.,\begin{equation}\int_0^1 tG(ts)ds=\int_0^t G(s)ds=\Big\{\int_0^t g(s)ds: g\in G\Big\}\subset C(I,X)\end{equation} is uniformly bounded by $\int_0^t\psi(s)ds$ for all $t\in I$.

Since $\digamma$ is a monotone continuous convex function on $V$,  it follows from (4.5), (4.6), Lemma 4.7,  Jensen's inequality (Lemma 4.6),
\begin{equation}\nonumber
 \begin{aligned}
\mu\Big(\int_0^t G(s)ds\Big)&=\digamma\Big(TQJ(\int_0^t G(s)ds)\Big)\\
&\leq\digamma\Big(\int_0^t TQJ(G(s))ds\Big)\\
&=\digamma\Big(\int_0^1 TQJ(tG(ts))ds\Big)\\
&\leq \int_0^1\digamma\Big(TQJ(tG(ts)\Big)ds\\
&=\int_0^1\mu\Big(tG(ts)\Big)ds\\
&=\frac{1}t\int_0^t\mu\Big(tG(s)\Big)ds.
\end{aligned}
\end{equation}
Therefore, (4.3) holds.

\end{proof}

\begin{corollary}
Let $X$ be a  Banach space,  $I=[0,a]$, $G\subseteq L_1(I,X)$ be a nonempty bounded subset, and  $\psi\in L_1(I,\mathbb{R}^+)$ such that $\sup_{g\in G}\|g(t)\|\leq \psi(t)\;\;a.e.\;t\in I$. Assume that $\mathbb J_G: I\rightarrow C_b(\Omega)$ is strongly measurable. Then the following inequality holds
\begin{equation}
\mu\Big(\int_0^t G(s)ds\Big)\leq \int_0^t \mu\Big(G(s)\Big)ds,\;\;\forall 0<t\leq a,
\end{equation}
if one of the following conditions is satisfied.

i) $0<t\leq\min\{1, a\};$

ii) $\mu$ is a sublinear measure of noncompactness.
\end{corollary}

\begin{proof}
By Theorem 4.8, it suffices to note that \[\mu\Big(tG(s)\Big)=\digamma\Big(\mathbb T(tG(s))\Big)\leq t\digamma\Big(\mathbb T(G(s))\Big),\;{\rm if}\; 0\leq t\leq1;\]  and that
\[\mu\Big(tG(s)\Big)=t\mu\Big(G(s)\Big), \forall\;t\geq0,\]  if $\mu$ is a sublinear MNC.
\end{proof}

\section{Strong measurability and weak measurability of  $\mathbb J_G$}
In this section, we will discuss measurability and weak measurability of the mapping $\mathbb J_G: I\rightarrow C_b(\Omega)$ for a bounded set $G\subset L_1(I, X)$. As a result, we show that if $G$ is separable in $L_1(I, X)$ and there is $u\in L_p(I, X)$ for some $0<p$ such that $\sup_{g\in G}\|g(t)\|<u(t)$ a.e., then $\mathbb J_G$ is weakly measurable. Making use of this result, we further prove that $\mathbb J_G$ is strongly measurable if $G$ is one of the following classical classes:

a) $G\subset C(I,X)$ is an equi-continuous subset;

b) $G$ is a equi-regulated subset of $R(I, X)$ (the Banach space of bounded functions on $I$ satisfying that $\lim_{t\rightarrow t_0^{\pm}}u(t)$ exist for all $u\in R(I,X)\;{\rm and}\;t_0\in I$) endowed with the sup-norm; and

c)  $G\subset L_1(I,X)$ is uniformly measurable.

\begin{lemma}
Let $\mathbf T$ be a Hausdorff topological space, and $C_b(\mathbf T)$ be the Banach space of all bounded continuous  functions on $\mathbf T$ endowed with the sup-norm $\|f\|=\sup_{t\in\mathbf T}|f(t)|$ for $f\in C_b(\mathbf T)$. Then the closed unit ball $B_{C_b(\mathbf T)^*}$ of the dual $C_b(\mathbf T)^*$ satisfies
$$B_{C_b(\mathbf T)^*}=w^*\text{-}\overline{\rm co}\{\pm\delta_t: t\in\mathbf T\},$$
where $\delta_t\in C_b(\mathbf T)^*$ is the evaluation functional defined for $f\in C_b(\mathbf T)$ by $\langle\delta_t,f\rangle=f(t)$, and $w^*\text{-}\overline{\rm co}(A)$ denotes the $w^*$-closed convex hull of $A$ in $C_b(\mathbf T)^*$.
\end{lemma}
\begin{proof}
Since for each $f\in C_b(\mathbf T)$, \[\|f\|=\sup_{t\in\mathbf T}|f(t)|=\sup_{t\in\mathbf T}|\langle\delta_t,f\rangle|,\]  $\{\pm\delta_t: t\in\mathbf T\}$ is a norming set of $C_b(\mathbf T)$. It follows from the separation theorem of convex sets in locally convex spaces. Indeed, suppose to the contrary, that there is \[\varphi\in B_{C_b(\mathbf T)^*}\setminus w^*\text{-}\overline{\rm co}\{\pm\delta_t: t\in\mathbf T\}.\] By the separation theorem in the locally convex space $(C_b^*(\mathbf T),w^*)$, there exists $f\in C_b(\mathbf T)=(C_b^*(\mathbf T),w^*)^*$  such that $$\|f\|\geq\langle\varphi,f\rangle>\sup_{\psi\in w^*\text{-}\overline{\rm co}\{\pm\delta_t: t\in\mathbf T\}}\langle\psi,f\rangle=\sup_{t\in\mathbf T}\langle\pm\delta_t,f\rangle=\|f\|,$$ and this is a contradiction.
\end{proof}
\begin{theorem}
Let $X$ be a Banach space, $I=[0,a]$, $G\subseteq L_1(I,X)$ be a separable subset satisfying that there is $u\in L_p(I,\R^+)$ for some $p>0$ such that $\sup_{g\in G}\|g(t)\|\leq u(t)$ for almost all $t\in I$. Then $\mathbb J_G: [0,a]\rightarrow C_b(\Omega)$
is weakly  measurable.
\end{theorem}

\begin{proof}
We want to prove that for each $\varphi\in C_b(\Omega)^*$, the real-valued function
$\langle\varphi,\mathbb J_G(\cdot)\rangle$ is measurable on $I$.

Since $G\subseteq L_1(I,X)$ is  separable, for every $x^*\in \Omega=B_{X^*}$, the real-valued function $\mathbb J_G(\cdot)(x^*)$ defined for $t\in I$ by \[\mathbb J_G(t)(x^*)=\sup_{g\in G}\langle x^*,g(t)\rangle\]
is measurable. Indeed, let $\{g_n\}$ be a dense sequence of $G$. Then
\[\mathbb J_{\{g_n\}}(t)(x^*)\equiv\sup_{n}\langle x^*,g_n(t)\rangle=\mathbb J_G(t)(x^*),\;{\rm for\;almost\;all\;}t\in I.\]
Thus, the measurability of $\mathbb J_G(\cdot)(x^*)$ follows from the measurability of $\mathbb J_{\{g_n\}}(\cdot)(x^*).$

 Assume that $\varphi\in C_b(\Omega)^*$ with $\|\varphi\|=1$. By Lemma 5.1, there is a net $\{\varphi_\iota\}\subset{\rm co}\{\pm\delta_{x^*}:x^*\in\Omega\}$ such that
$\varphi_\iota\longrightarrow\varphi$ in the $w^*$-topology of $C_b(\Omega)^*$. Therefore, \begin{equation}\langle\varphi_\iota,\mathbb J_G(t)\rangle\longrightarrow\langle\varphi,\mathbb J_G(t)\rangle,\;t\in I.\end{equation}
Since for each $x^*\in \Omega\equiv B_{X^*}$, $$\langle\delta_{x^*},\mathbb J_G(t)\rangle=\langle \mathbb J_G(t),x^*\rangle=\sup_{g\in G}\langle x^*,g(t)\rangle$$ is measurable, $\langle\varphi_\iota,\mathbb J_G(t)\rangle$ and ${\rm sgn}(\langle\varphi_\iota,\mathbb J_G(t)\rangle)$ are measurable for each $\varphi_\iota$.  Let $E=\{t\in I: \langle\varphi,\mathbb J_G(t)\rangle\neq0\}$. Then
 $$sgn(\langle\varphi_\iota,\mathbb J_G(t)\rangle)\cdot\chi_E\longrightarrow sgn(\langle\varphi,\mathbb J_G(t)\rangle),\;\;t\in I.$$  On the other hand, since $\{sgn(\langle\varphi_\iota,\mathbb J_G(\cdot)\rangle)\}\subset L_2(I,\R)$ is a bounded net, there is a subnet of it weakly convergent to an element $v\in L_2(I,\R)$. Clearly, \begin{equation}v(t)\chi_E(t)=sgn(\langle\varphi,\mathbb J_G(t)\rangle)\;{\rm for \;almost\; all\;} t\in I.\end{equation}
%Therefore, $sgn(\langle\varphi,\mathbb J_G(\cdot)\rangle)$ is measurable.

Let $u_\iota=|\langle\varphi_\iota,\mathbb J_G(\cdot)\rangle|$.  Since \[|\langle\varphi_\iota,\mathbb J_G(t)\rangle|\leq\|\mathbb J_G(t)\|=\sup_{x^*\in\Omega, g\in G}|\langle x^*,g(t)\rangle|=\sup_{ g\in G}\|g(t)\|\leq u(t) \;a.e.,\]
and since $u\in L_p(I,\R)$, $\{u_\iota^{\frac{p}2}\}\subset L_2(I,\R)$ is a bounded net which is pointwise  convergent to $|\langle\varphi,\mathbb J_G(t)\rangle|^{\frac{p}{2}}$ for $t\in I$. This entails that $\{u_\iota^{\frac{p}2}\}$ admits a subnet weakly convergent to some element $w\in L_2(I,\R)$. Clearly, $w(t)=|\langle\varphi,\mathbb J_G(t)\rangle|^{\frac{p}{2}}$ for almost all $t\in I$. Therefore, $|\langle\varphi,\mathbb J_G(\cdot)\rangle|^{\frac{p}{2}}\in L_2(I,\R)$. Consequently, $|\langle\varphi,\mathbb J_G(\cdot)\rangle|$ is measurable. Hence, $\chi_E$ is measurable. This and (5.2) entail that $sgn(\langle\varphi,\mathbb J_G(\cdot)\rangle)$ is measurable. Since \[\langle\varphi,\mathbb J_G(\cdot)\rangle=sgn(\langle\varphi,\mathbb J_G(\cdot)\rangle)|\langle\varphi,\mathbb J_G(\cdot)\rangle|,\]  $\langle\varphi,\mathbb J_G(\cdot)\rangle$ is measurable.
\end{proof}

\begin{theorem}
Let $G\subset C(I,X)$ be a nonempty equi-continuous subset. Then $\mathbb J_G: I\rightarrow C_b(\Omega)$ is continuous, hence, strongly measurable.
\end{theorem}
\begin{proof}
It follows from the equi-continuity of $G$ on $I$ that for all $\eps>0$ and $t_0\in I$, there is $\delta>0$ such that
\[t\in I, |t-t_0|<\delta\Longrightarrow \|g(t)-g(t_0)\|<\eps,\;\forall\;g\in G.\]
Thus, $t\in I, |t-t_0|<\delta$ imply
\[\|\mathbb J_{G}(t)-\mathbb J_{G}(t_0)\|=\Big|\sup_{x^*\in\Omega, g\in G}\langle x^*, g(t)\rangle-\sup_{x^*\in\Omega, g\in G}\langle x^*, g(t_0)\rangle\Big|\]
\[\leq\Big|\sup_{x^*\in\Omega, g\in G}\langle x^*, g(t)-g(t_0)\rangle\Big|=\sup_{g\in G}\|g(t)-g(t_0)\|\leq\eps.\]
\end{proof}
The next result follows from the theorem above and Corollary 4.9.
\begin{corollary}
Let $X$ be a  Banach space, $\mu$ be a convex MNC on $X$, $I=[0,a]$, and $G\subseteq C(I,X)$ be a nonempty equi-continuous subset.   Then the following inequality holds
\begin{equation}
\mu\Big(\int_0^t G(s)ds\Big)\leq \frac{1}t\int_0^t \mu\Big(tG(s)\Big)ds,\;\;\forall 0<t\leq a.
\end{equation}
In particular,
\begin{equation}
\mu\Big(\int_0^t G(s)ds\Big)\leq\int_0^t \mu\Big(G(s)\Big)ds,\;\;\forall 0<t\leq a.
\end{equation}
if one of the following conditions is satisfied.

i) $0<t\leq\min\{1, a\};$

ii) $\mu$ is a sublinear MNC.
\end{corollary}
\begin{remark}
Corollary 5.4 is a generalization of K. Goebel  et al. In 1970, K. Goebel and W. Rzymowski \cite{goe} showed that the inequality (5.4) holds for equi-continuous subsets $G\subset C(I,X)$ and for the Hausdorff MNC $\beta$, i.e. when $\mu=\beta$. (See, also, \cite{rz}.) In 1980, J. Bana\'{s} and  K. Goebel \cite{bag} showed the inequality (5.4) holds for equi-continuous subsets $G\subset C(I,X)$ and for homogeneous MNC $\mu$.
\end{remark}
Before stating next result, we recall the notions of regulated functions and of equi-regulated sets of such functions (see, for example, \cite{ols}).
\begin{definition}
i) A function $f: [a,b]\rightarrow X$ is said to be regulated provided for every $t\in[a,b)$ the right-sided limit $\lim_{s\rightarrow t^+}f(s)\equiv f(t^+)$ exists and for every $t\in(a,b]$ the left-sided limit $\lim_{s\rightarrow t^-}f(s)\equiv f(t^-)$ exists.

We denote by $R([a,b], X)$ the Banach space of all regulated functions defined on the interval $[a,b]$ endowed with the sup-norm.

ii) A nonempty subset $G\subset R(I,X)$ is called equi-regulated if
\[\forall t\in(a,b], \forall \eps>0,\; \exists \delta>0,\;\forall g\in G,\; \forall t_1,t_2\in(t-\delta,t)\cap[a,b], \|g(t_2)-g(t_1)\|\leq\eps, \]
\[\forall t\in[a,b), \forall \eps>0,\; \exists \delta>0,\;\forall g\in G,\; \forall t_1,t_2\in(t,t+\delta)\cap[a,b], \|g(t_2)-g(t_1)\|\leq\eps. \]
\end{definition}

\begin{theorem}
Let $X$ be a Banach space and $I=[0,a]$. Assume that $G\subset R(I, X)$ is a separable equi-regulated set. Then the mapping
$\mathbb J_G: I\rightarrow C_b(\Omega)$ is strongly measurable.
\end{theorem}
\begin{proof}
By the definition of regulated functions and equi-regulated sets, it is easy to see that $G$ is separable in $L_1(I,X)$ and  bounded in $R(I,\R)$, i.e. \[\sup_{g\in G}\|g\|_{R(I,\R)}=\sup_{g\in G, t\in I}\|g(t)\|_X=c<\infty.\]
By Theorem 5.2, $\mathbb J_G$ is weakly measurable.
To prove that $\mathbb J_G$ is strongly measurable, it suffices to show that $\mathbb J_G(I)$  is separable in $C_b(\Omega)$.

Through a routine argument of compactness of $I$ similar to  L. Olszowy and T. Zajac \cite[Th. 3.1]{ols}, we obtain that
for all $\eps>0$, there is an  open cover of $I$ consisting of finitely many open intervals of the form
\[\Big\{[0,\delta_0), (s_1-\delta_1, s_1+\delta_1),\cdots, (s_{n-1}-\delta_{n-1}, s_{n-1}+\delta_{n-1}), (s_n, a]\Big\},\]
such that
\[\sup_{g\in G}\|g(t_2)-g(t_1)\|<\eps,\;\forall t_1,t_2\in[0,\delta_0), {\rm or,}\;(s_n, a],\;\;\;\;\;\;\;\;\;\;\;\;\;\; \]
\[\sup_{g\in G}\|g(t_2)-g(t_1)\|<\eps,\;\forall t_1,t_2\in(s_1-\delta_j, s_j), {\rm or,}\;(s_j, s_j+\delta_j),\]
 \[{\rm where\;\;}j=1,2,\cdots, n-1,\;{\rm for\; some\;} n\in\N;\;
  \delta_j>0,  \;j=0,1,\cdots,n-1;\;\;\;\;\;\;\;\;\;\;\;\;\;\;\;\;\;\;\;\;\;\;\;\;\;\;\;\;\;\;\;\;\;\;\;\;\;\;\;\;\;\;\;\;\;\;\;\;\;\;\;\;\;\;\;\;\;\;\;\;\;\;\;\;\;\;\;\;\;\;\;\;\]
  \[{\rm and\;\;}\delta_0, s_n, s_j, s_j-\delta_j, s_j+\delta_j\in(0,a),\; j=1,\cdots,n-1.\;\;\;\;\;\;\;\;\;\;\;\;\;\;\;\;\;\;\;\;\;\;\;\;\;\;\;\;\;\;\;\;\;\;\;\;\]
Put
\[I_0^+=[0,\delta_0),\;\;  I_n^+=(s_n, a],\]
\[I_j^+=(s_j, s_j+\delta_j),\;I_j^-=(s_1-\delta_j, s_j), \;  j=1,2,\cdots, n-1.\]
Therefore,
\[\|\mathbb J_G(t_2)-\mathbb J_G(t_1)\|\leq\sup_{g\in G, x^*\in\Omega}\langle x^*,g(t_2)-g(t_1)\rangle\]\[\;\;\;\;\;\;\;\;\;\;\;\;\;\;\;\;\;\;\;\;\;\;=\sup_{g\in G}\|g(t_2)-g(t_1)\|<\eps,\]
whenever $t_1, t_2\in I_j^+,\;j=0,1,\cdots,n$, or, $t_1, t_2\in I_j^-,\;j=1, 2\cdots,n-1.$
Thus, each element of $\mathbb J_G(I_j^+),\;j=0,1,\cdots,n$ and $\mathbb J_G(I_j^-),\;j=1,2,\cdots,n-1$ is contained in a ball of $C_b(\Omega)$ with radius $\eps$.  Consequently,
$\mathbb J_G(I)$ contained in the union of $3n$ balls of $C_b(\Omega)$ with radius $\eps$. This entails that $\mathbb J_G(I)$ is a relatively compact  subset of $C_b(\Omega)$, hence, it is separable.
\end{proof}
The result below follows from the theorem above and Corollary 4.9.
\begin{corollary}
Let $X$ be a  Banach space, $\mu$ be a convex MNC on $X$, $I=[0,a]$, and $G\subseteq R(I,X)$ be a nonempty separable equi-regulated subset.   Then the following inequality holds
\begin{equation}
\mu\Big(\int_0^t G(s)ds\Big)\leq \frac{1}t\int_0^t \mu\Big(tG(s)\Big)ds,\;\;\forall 0<t\leq a.
\end{equation}
In particular,
\begin{equation}
\mu\Big(\int_0^t G(s)ds\Big)\leq\int_0^t \mu\Big(G(s)\Big)ds,\;\;\forall 0<t\leq a.
\end{equation}
if one of the following conditions is satisfied.

i) $0<t\leq\min\{1, a\};$

ii) $\mu$ is a sublinear MNC.
\end{corollary}
\begin{remark}
Corollary 5.8 is a generalization of L. Olszowy and T. Zajac \cite[Th. 3.1]{ols} in 2020, where they showed that the inequality (5.6) holds for the Hausdorff MNC $(\mu=)\beta$.
\end{remark}

For a finite measure space $(\Gamma,\Sigma,\eta)$ and a Banach space $X$, we denote by $L_1(\Gamma,\eta, X)$ the Banach space of all $X$-valued $\eta$-integrable functions $f$ endowed with the $L_1$-norm $\|f\|=\int_\Gamma \|f(\gamma)\|d\eta$. In particular, if $\Gamma=I=[a,b]\subset\R$ and $\eta$ is  the Lebesgue measure on $\R$, then simply denote it by $L_1(I,X)$.

The following useful notion of uniform measurability of functions was introduced by M. Kunze and G. Schl\"{u}chtermann \cite[Def. 3.16]{kun}.
\begin{definition}
A bounded set $G\subset L_1(\Gamma, \eta, X)$ is said to be uniformly $\eta$-measurable provided for all $\eps>0$ and $A\in\Sigma$ there exist $A_\eps\in\Sigma$ and mutually disjoint $A_1, \cdots A_n\in\Sigma$ with
$\cup_{j=1}^n A_j=A_\eps$ and with $\eta(A\setminus A_\eps)<\eps$ such that for $j=1,2,\cdots,n$ we can choose $\gamma_j\in A_j$ satisfying
\begin{equation}
\sup_{g\in G}\|g(\cdot)\chi_{A_\eps}-\sum_{j=1}^ng(\gamma_j)\chi_{A_j}\|_{L_1}<\eps.
\end{equation}
\end{definition}

Given a subset $G\subset L_0(\Gamma,\eta, X)$ of $X$-valued $\eta$-measurable functions, we denote by $G(\gamma)=\{g(\gamma):g\in G\},\,\gamma\in\Gamma$; and let $\mathbb J_G: \Gamma\rightarrow C_b(\Omega)$ be again defined for $\gamma\in\Gamma$ by
\begin{equation}
\mathbb J_G(\gamma)(x^*)=J\Big(G(\gamma)\Big)(x^*)=\sup_{g\in G}\langle x^*,g(\gamma)\rangle=\sigma_{G(\gamma)}(x^*),\;x^*\in\Omega\equiv B_{X^*}.
\end{equation}

\begin{theorem}
Let $G\subset L_0(\Gamma,\eta, X)$ be a  subset with \begin{equation}\sup_{g\in G}\|g(\gamma)\|<\infty,\;{\rm for\;almost\;all}\;\gamma\in\Gamma.\end{equation}
If $G$ is uniformly $\eta$-measurable, then $\mathbb J_G: \Gamma\rightarrow C_b(\Omega)$ is strongly measurable.
\end{theorem}

\begin{proof}
Without loss of generality, we can assume that (5.9) holds for all $\gamma\in\Gamma$.

 Given $m\in\mathbb N$, let $A^1=\Gamma$, and $\eps=1/m$. It follows from Definition 5. 10 that there exist $\Gamma^1_m\in\Sigma$ and mutually disjoint $A^1_{m,1}, \cdots A^1_{m,n_{m,1}}\in\Sigma$, $\gamma^1_j\in\Gamma$ for $j=1,2,\cdots,n_{m,1}$ with
$\cup_{j=1}^{n_{m,1}} A^1_{m,j}=\Gamma^1_m$ and with $\eta(\Gamma\setminus \Gamma^1_m)<1/m$ such that
\begin{equation}
\sup_{g\in G}\|g(\cdot)\chi_{\Gamma^1_{m}}-\sum_{j=1}^{n_{m,1}}g(\gamma^1_j)\chi_{A^1_{m,j}}\|<1/m.
\end{equation}

Next, we substitute $A^2\equiv\Gamma\setminus\Gamma^1_m$ for $A^1(=\Gamma)$ in the procedure above. Then there exist $\Gamma^2_m\in\Sigma$ and mutually disjoint $A^2_{m,1}, \cdots A^2_{m,n_{m,2}}\in\Sigma$, $\gamma^2_j\in\Gamma$ for $j=1,2,\cdots,n_{m,2}$ with
$\cup_{j=1}^{n_{m,2}} A^2_{m,j}=\Gamma^2_m$ and with $\eta(A^2\setminus\Gamma^2_m)<1/m$ such that
\begin{equation}
\sup_{g\in G}\|g(\cdot)\chi_{\Gamma^2_{m}}-\sum_{j=1}^{n_{m,2}}g(\gamma^2_j)\chi_{A^2_{m,j}}\|<1/m.
\end{equation}
We can claim that $\Gamma^2_m\subset A^2$; Otherwise, we can replace $\Gamma^2_m$ by $\Gamma^2_m\cap A^2$.

Inductively, we obtain a  sequence $\{\Gamma_m^n\}_{n=1}^\infty$ of mutually disjoint measurable sets with $\eta(\Gamma\setminus\cup_{n=1}^\infty \Gamma_m^n)=0$ and a family of strongly measurable functions $\{g_{f,m}: f\in M\}$:
$f_{g,m}: \Gamma_m\equiv\cup_{n=1}^\infty \Gamma_m^n\rightarrow X$ defined for $\gamma$ by
\begin{equation}
f_{g,m}(\gamma)=\sum_{j=1}^{n_{m,l}}g(\gamma^l_j)\chi_{A^l_{m,j}}, \;\gamma\in \Gamma_m^l, l\in\N,
\end{equation}
which satisfies
\begin{equation}
\sup_{g\in G, \gamma\in\Gamma_m}\|g(\gamma)\chi_{\Gamma_{m}}(\gamma)-f_{g,m}(\gamma)\|<1/m,
\end{equation}
and $\eta(\Gamma\setminus\Gamma_m)=0$. Clearly, for each $g\in G$ and $l\in\N$, $f_{g,m}$ is a simple function on $\Gamma_m^l$. Let $G_{m,l}=\{f_{g,m}|_{\Gamma^l_m}: g\in G\}.$
For each $l\in\N$, by (5.12) we know that
%$$\cup_{f\in M, \gamma\in\Gamma^l_m}\{g_{f,m}(\gamma)\}=\cup_{f\in M, \gamma\in\Gamma^l_m}\{\sum_{j=1}^{n_{m,l}}f(\gamma^l_j)\chi_{A^l_{m,j}}\}\subset X,$$
$$J_{m,l}(\gamma)(x^*)\equiv\sup_{g\in G}\Big\{\langle x^*,f_{g,m}|_{\Gamma^l_m}(\gamma)\rangle\Big\},\;x^*\in\Omega\; (=B_{X^*})$$
defines a simple function $J_{m,l}: \Gamma_m^l\rightarrow C_b(\Omega)$. %Let
%$J_m:  \Gamma_m\rightarrow C_b(\Omega)$ be defined for $\gamma\in\Gamma_m$ by
%$$J_{m}(\gamma)(x^*)\equiv\sup_{f\in M}\langle x^*,f(\gamma)\rangle,\;x^*\in\Omega.$$
Note that for $\gamma\in\Gamma^l_m$,
\[\|\mathbb J_G(\gamma)-J_{m,l}(\gamma)\|_{C_b(\Omega)}\equiv\sup_{x^*\in\Omega}(\mathbb J_G(\gamma)(x^*)-J_{m,l}(\gamma)(x^*)).\]
$\forall \eps>0$ and $\gamma\in\Gamma^l_m$,  let $x^*_\eps\in\Omega$ be such that
\begin{equation}\|\mathbb J_G(\gamma)-J_{m,l}(\gamma)\|_{C_b(\Omega)}<\mathbb J_G(\gamma)(x_\eps^*)-J_{m,l}(\gamma)(x_\eps^*)+\eps.\end{equation}
Then there is $g_\eps\in G$ be such that
$\mathbb J_G(\gamma)(x_\eps^*)<\langle x_\eps^*,g_\eps(\gamma)\rangle+\eps$. Therefore,
\[\mathbb J_G(\gamma)(x_\eps^*)-J_{m,l}(\gamma)(x_\eps^*)+\eps\;\;\;\;\;\;\;\;\;\;\;\;\;\;\;\]
\[<(\langle x_\eps^*,g_\eps(\gamma)\rangle+\eps)-\langle x_\eps^*,f_{g_\eps,m}(\gamma)\rangle\]
\[\leq\|g_\eps(\gamma)-f_{g_\eps,m}(\gamma)\|+\eps<\frac{1}m+\eps.\]
Since $\eps$ is arbitrary, the inequalities above  and (5.14) imply that for all $\gamma\in\Gamma_m^l$,
\begin{equation}\|\mathbb J_G(\gamma)-J_{m,l}(\gamma)\|_{C_b(\Omega)}\leq\frac{1}m.\end{equation}

Now, we define $J_m: \Gamma_m\rightarrow C_b(\Omega)$ by $J_m(\gamma)=J_{m,n}(\gamma), \;\gamma\in\Gamma_m^n, n=1,2,\cdots.$ Clearly, $J_m$ is strongly measurable and it satisfies
\begin{equation}\|\mathbb J_G(\gamma)-J_{m}(\gamma)\|_{C_b(\Omega)}\leq\frac{1}m,\;\gamma\in\Gamma_m,\end{equation}
and
$J_m(\Gamma_m)=\cup_{n=1}^\infty J_m(\Gamma_m^n)$ is a countable subset of $C_b(\Omega)$.

By taking $m=1,2,\cdots$, we can obtain the following two sequences:

a) a sequence $\{\Gamma_m\}$ of measurable sets satisfying $\eta(\Gamma\setminus\Gamma_m)=0$ for all $m\in\N$;

b) a sequence  $\{J_m\}$ of strongly measurable functions $J_m$ defined on $\Gamma_m$ with countable ranges satisfying
\begin{equation}\|\mathbb J_G(\gamma)-J_{m}(\gamma)\|_{C_b(\Omega)}\leq\frac{1}m,\;\gamma\in\Gamma_m, m\in\N.\end{equation}

Finally, let $\Gamma_0=\bigcap_{m=1}^\infty\Gamma_m$. Then $\eta(\Gamma\setminus\Gamma_0)=0$, and
$$\lim_{m\rightarrow\infty}\|\mathbb J_G(\gamma)-J_m(\gamma)\|_{C_b(\Omega)}=0,\;\forall\;\gamma\in\Gamma_0.$$
Consequently, $\mathbb J_G:\Gamma\rightarrow C_b(\Omega)$ is strongly measurable.
\end{proof}
The following result is a consequence of  Theorem 5.11 and Corollary 4.9.
\begin{corollary}
Let $X$ be a  Banach space, $\mu$ be a convex MNC on $X$, $I=[0,a]$, and $G\subseteq L_1(I,X)$ be a nonempty separable and uniformly measurable subset.   Then the following inequality holds
\begin{equation}
\mu\Big(\int_0^t G(s)ds\Big)\leq \frac{1}t\int_0^t \mu\Big(tG(s)\Big)ds,\;\;\forall 0<t\leq a.
\end{equation}
In particular,
\begin{equation}
\mu\Big(\int_0^t G(s)ds\Big)\leq\int_0^t \mu\Big(G(s)\Big)ds,\;\;\forall 0<t\leq a,
\end{equation}
if one of the following conditions is satisfied,

i) $0<t\leq\min\{1, a\};$

ii) $\mu$ is a sublinear MNC.
\end{corollary}

\begin{remark} Corollary 5.12 is a generalization of Kunze and G. Schl\"{u}chtermann  \cite[Corollary 3.19]{kun}, where they showed that the inequality (5.19) holds for the Hausdorff MNC $\beta$, i.e. $\mu=\beta$.
\end{remark}

\section{On representation of measures of non-property $\mathfrak{D}$}
%The ball measure of non-property $\mathfrak{D}$ is introduced in \cite{cheng1}.
In this section, we will discuss a more general class of measures on Banach spaces, namely, convex MNP$\mathfrak{D}$ (defined in Definition 1.5), which contains convex MNCs, convex MNWCs, convex MNSWCs, convex MNRNPs and convex MNAs as its special cases. We will show the results presented in Sections 3-5 for MNCs are still true for these measures of non generalized compactness.

In this section, we again use $\mathscr D(X)$ (or, simply, $\mathscr D$) to denote a fundamental sub-semigroup of  $\mathscr C(X)$ (Definition 1.3).

\begin{proposition}
Every fundamental closed sub-semigroup $\mathscr{D}$ of  $\mathscr{C}(X)$ contains all nonempty compact convex subsets of $X$. i.e.  $\mathscr{D}(X)\supset\mathscr{K}(X)$.
\end{proposition}
\begin{proof} Since $\mathscr{C}(X)$ is complete, and since $\mathscr{D}$ is closed in $\mathscr{C}(X)$, $\mathscr{D}$ is necessarily complete.
On the other hand, it follows from Definition 1.3 that every fundamental closed sub-semigroup $\mathscr{D}$ of  $\mathscr{C}(X)$ contains all nonempty convex polyhedrons of $X$, i.e. for any finite subset $F\subset X$,  ${\rm co}(F)\in\mathscr{D}.$ Note that the set $\mathscr{P}(X)$ of all nonempty convex polyhedrons of $X$ is dense in  $\mathscr{K}(X)$.  Completeness of $\mathscr{D}$ entails $\mathscr{D}(X)\supset\mathscr{K}(X)$.
\end{proof}

\begin{definition}
A convex measure of non-property $\mathfrak{D}$ on $X$ is said to be, successively, a convex measure of noncompactness, of non-weak compactness, of non-super weak compactness, of non-Radon-Nikod\'ym property, and of non-Asplundness if we take, successively,

$ \mathscr{D}=\mathscr{K}(X)$, the fundamental semigroup consisting of all nonempty convex compact subsets  of $X$;

$\mathscr{D}=\mathscr{W}(X),$ the fundamental semigroup consisting of all nonempty convex weakly compact subsets of $X$;

$\mathscr{D}=$sup-$\mathscr{W}(X)$, the fundamental semigroup consisting of all nonempty convex super weakly compact subsets of $X$;

$\mathscr{D}=\mathscr{R}(X)$, the fundamental semigroup consisting of all nonempty closed bounded convex  subsets of $X$ with the Radon-Nikod\'ym property; and

$\mathscr{D}=\mathscr{A}(X)$, the fundamental semigroup consisting of all nonempty closed bounded convex Asplund subsets of $X$.
\end{definition}
\begin{definition}
i) Kuratowski's measure $\alpha_{\mathscr{D}}$ of non-property $\mathfrak{D}$ on $X$ is defined  for $B\in \mathscr{B}(X)$ by
$$\alpha_{ \mathscr{D}}(B)=\inf\Big\{r>0:\;\exists D\in\mathscr{D},  \{E_j\}_{j\in F}\subset\mathscr B(X) \;{s. \;t.}\;B\subset D+\cup_{j\in F} E_j \Big\},$$
where $F\subset\N$ is a finite set with $ {\rm diam}(E_j)\leq r$, for all $ j\in F$.

ii) \cite{ccs} \;The Hausdorff measure $\beta_{\mathscr{D}}$ of non-property $\mathfrak{D}$ on $X$ is defined  for $B\in \mathscr{B}(X)$ by
$$\beta_{\mathscr{D}}(B)=\inf\Big\{r>0:\;\exists D\in\mathscr{D},  \;{s. \;t.}\;B\subset D+rB_X  \Big\}.$$
\end{definition}

\begin{theorem}
Let $X$ be a Banach space, and $\mu:\mathscr B(X)\rightarrow\mathbb R^+$ be a convex measure of non-property $\mathfrak{D}$. Then

i)  {\rm [{\bf Density determination}]} $\mu(\overline{B})=\mu(B),\;\forall B\in \mathscr B(X)$;

ii) {\rm[{\bf Translation invariance}]} $\mu(B+D)=\mu(B),\;\forall  B\in \mathscr B(X), D\in\mathscr D(X)$;

iii) {\rm[{\bf Negiligbility}]} $\mu(B\cup D)=\mu(B),\;\forall \; B\in \mathscr B(X), D\in\mathscr D(X)$.

\noindent
If, in addition, one of the following four conditions is satisfied:

a)  $\mathscr {D}=\mathscr K;$

b)  $\mathscr{D}=\mathscr W;$

c)  $\mathscr{D}=$sup-$\mathscr W;$

d) $X$ is reflexive,

\noindent
then

iv) {\rm[{\bf Continuity}]} $\emptyset\neq C_{n+1}\subset C_n\in\mathscr C(X),\;n\in\N;\; \mu(C_n)\rightarrow0\Longrightarrow$\[\bigcap_n{C_n}\neq\emptyset.\]
\end{theorem}
\begin{proof}
 By Proposition 6.1, $\mathscr{D}\supset \mathscr{K}$. Therefore, i), ii) and iii) follow from an argument which is the same as the proof of Theorem 3.4 but substituting ``convex MNP$\mathfrak{D}$" for ``convex MNC". It remains to show iv).

 If a) $\mathscr{D}=\mathscr K,$ then $\mu$ is a convex MNC. Thus, iv) follows from Theorem 3.4.

 If b) $\mathscr{D}=\mathscr W,$ then $\mu$ is a convex MNWC. For each $n\in\N$, choose any $x_n\in C_n$ and let $W_n=\overline{\rm co}\{x_k\}_{k\geq n}$. Then
 \[\mu(W_1)=\mu(W_n)\leq\mu(C_n)\rightarrow0,\;\;{\rm as\;\;}n\rightarrow\infty.\]
Thus, $\mu(W_1)=0$, which entails that $W_1$ is weakly compact.  Weak compactness and non-inceasing monotonicity of $\{W_n\}$ imply
\[\emptyset\neq\bigcap_n{W_n}\subset\bigcap_n{C_n}.\]

 If c) $\mathscr{D}=$sup-$\mathscr W,$ then by b), it suffices to note that every super weakly compact set is weakly compact.

 If d) $X$ is reflexive, then every $C_n$ is nonempty weakly compact. Consequently, $\bigcap_n{C_n}\neq\emptyset$.
\end{proof}
The following representation theorem of convex MNP$\mathfrak{D}$ is an analogy of Theorem 3.11.

\begin{theorem} Suppose that $X$ is a Banach space and that $\mathscr{D}$ is a fundamental sub-semigroup of $\mathscr{C}(X)$. Then there is a Banach space $C(K)$ for some compact Hausdorff space $K$ such that for every $\mu$ which is a convex measure of non property $\mathfrak{D}$ with respect to  $\mathscr{D}$ on $X$, there is a function $\digamma=\digamma_{\mathscr{D}}$ on the cone $V\equiv\mathbb T(\mathscr{B}(X))\subset C(K)^+$ satisfying

i) $\mu(B)=\digamma(\mathbb TB)$, for all $B\in\mathscr{B}(X)$;

ii)  $\digamma$ is nonnegative  convex and monotone on $V$;

iii) $\digamma$ is bounded by $b_r=\digamma(r\mathbb TB_X)$ on $V\bigcap(rB_{C(K)})$, for all $r\geq0$;

iv) $\digamma$ is $c_r$-Lipschitian on $V\bigcap(rB_{C(K)})$, for all $r\geq0$, where $c_r=\digamma\big((1+r)\mathbb TB_X\big)=\mu\big((1+r)B_X\big)$;

v) In particular, if $\mu$ is a sublinear of non property $\mathfrak{D}$, then we can  take $c_r=\mu(B_X)$ in iv).
\end{theorem}
\begin{proof}
Let $\mu$  be a convex measure of non property $\mathfrak{D}$ with respect to  $\mathscr{D}$ on $X$. Let
\[\digamma_{\mathscr D}(\mathbb T(B))=\mu(B),\;\;\forall\;B\in\mathscr B(X).\]
By Proposition 6.1, we know that all the propositions, lemmas and theorems in Section 3 hold again if we substitute $\digamma_{\mathscr D}$ for the convex function $\digamma$. In particular, Theorem 3.11 is still true for $\digamma_{\mathscr D}$.
\end{proof}
The following the theorem is an analogy of Theorem 4.8.
\begin{theorem}
Let $\mu$ be a convex measure of non property $\mathfrak{D}$ on  a Banach space $X$,  $I=[0,a]$, $G\subseteq L_1(I,X)$ be a nonempty bounded subset, and  $\psi\in L_1(I,\mathbb{R}^+)$ such that $\sup_{g\in G}\|g(t)\|\leq \psi(t)\;\;a.e.\;t\in I$. Assume that $\mathbb J_G: I\rightarrow C_b(\Omega)$ is strongly measurable. Then $\mu\Big(G(t)\Big)$ is measurable and
\begin{equation}
\mu\Big(\int_0^t G(s)ds\Big)\leq \frac{1}t\int_0^t \mu\Big(tG(s)\Big)ds,\;\;\forall 0<t\leq a.
\end{equation}
\end{theorem}

\begin{corollary}
Let $X$ be a  Banach space, $\mu$ be a convex measure of non property $\mathfrak{D}$ on  $X$,  $I=[0,a]$, $G\subseteq L_1(I,X)$ be a nonempty bounded subset, and  $\psi\in L_1(I,\mathbb{R}^+)$ such that $\sup_{g\in G}\|g(t)\|\leq \psi(t)\;\;a.e.\;t\in I$. Assume that $\mathbb J_G: I\rightarrow C_b(\Omega)$ is strongly measurable. Then the following inequality holds
\begin{equation}
\mu\Big(\int_0^t G(s)ds\Big)\leq \int_0^t \mu\Big(G(s)\Big)ds,\;\;\forall 0<t\leq a.
\end{equation}
if one of the following conditions is satisfied.

i) $0<t\leq\min\{1, a\};$

ii) $\mu$ is a sublinear measure of non property $\mathfrak{D}$.
\end{corollary}

The next result follows from  Corollary 6.7, Theorems 5.3, 5.7 and 5.11.
\begin{corollary}
Let $X$ be a  Banach space, and  $\mu$ be a convex measure of non property $\mathfrak{D}$ on $X$, $I=[0,a]$. If $G$ satisfies one of the following conditions

i) $G\subseteq C(I,X)$ is a nonempty equi-continuous subset;

ii) $G\subseteq R(I,X)$ is a nonempty separable equi-regulated subset;

iii) $G\subseteq L_1(I,X)$ is a nonempty separable uniformly measurable set,

\noindent
then the following inequality holds
\begin{equation}
\mu\Big(\int_0^t G(s)ds\Big)\leq \frac{1}t\int_0^t \mu\Big(tG(s)\Big)ds,\;\;\forall 0<t\leq a.
\end{equation}
In particular,
\begin{equation}
\mu\Big(\int_0^t G(s)ds\Big)\leq\int_0^t \mu\Big(G(s)\Big)ds,\;\;\forall 0<t\leq a.
\end{equation}
if one of the following conditions is satisfied.

a) $0<t\leq\min\{1, a\};$

b) $\mu$ is a sublinear measure of non property $\mathfrak{D}$.
\end{corollary}

\section{Some applications to a Cauchy problem}

In this section, as applications of the representation theorem of convex MNC (Theorem 3.11) and the inequalities in Section 4,  we will give two examples of solvability  the Cauchy problem (1.5) in the first section, i.e.
\begin{equation} \left\{\begin{array}{cc}
                    x'(t)=f(t, x), & a\geq t>0; \\
                    x(0)=x_0~~~~~~~~
                  \end{array}
\right.
\end{equation}
in a Banach space $X$. This results can be regarded as an extension of the classical result of K. Goebel and W. Rzymowski \cite{goe} (see, also,  \cite{rz}), J. Bana\'{s} and K. Goebel \cite{bag}. Our proof is based on their nice and concise constructions but with slight modifications.

The first example (Theorem 7.3) is to show that (7.1) has at least solution on the interval $I=[0,a]$  under the assumptions (1) $f: I\times B(x_0,r)\rightarrow X$ is uniformly continuous and bounded, and (2) for every nonempty bounded set $B\in \mathscr{B}(X)$, we have $\mu(f(t,B))\leq w(t,\mu(B))$, where $w(t,u)$ is a Kamke function and $\mu$ is a convex MNC. The same result has been  proven  by J. Bana\'{s} and K. Goebel \cite{bag} but under the assumption that $\mu$ is a ``symmetric" sublinear MNC.

Let $w(t,u)$  be a real nonnegative function defined on $(0,a]\times [0,+\infty)$ which is continuous with respect to $u$ for any fixed $t$ and measurable with respect to $t$ for each fixed $u$.  We call $w(t,u)$ a Kamke comparison function if $w(t,0)=0$ and the constant function $u=0$ is the unique solution of the integral inequality $$u(t)\leq\int_0^t w(s,u(s))ds,\;{\rm for}\;t\in (0,a]$$
which satisfies the condition $\lim\limits_{t\rightarrow 0^+}\frac{u(t)}{t}=\lim\limits_{t\rightarrow 0^+}u(t)=0$.

For $x\in C(I,X)$ and $B\subset C(I,X)$, where $I=[0,a]$, $X$ is a Banach space, the modulus of continuity of $x$ (resp., $B$) is defined by
\begin{equation}\omega(x,\eps)=\sup\{\;\|x(t)-x(s)\|:t,s\in [0,a],|t-s|\leq\eps\}.\end{equation}
\begin{equation}{\rm (resp.,}\;\omega(B,\eps)=\sup\{\omega(x,\eps):x\in B\}.{\rm )}\end{equation}
The following result is an extension of J. Bana\'{s} and K. Goebel \cite[Lemma 13.2.1]{bag}.
\begin{lemma}
Let $X$ be a Banach space, $I=[0,a]$, $\mu$ be a convex MNC on $X$ and $B$ be a nonempty bounded subset of $C(I,X)$. Then
\begin{equation}
|\;\mu(B(t))-\mu(B(s))\;|\leq L_B\cdot\omega(B,|t-s|),\;\;\forall t,s\in I,
\end{equation}
where $L_B>0$ is a constant. If   $B$ is equicontinuous, then $\mu(B(t))$ is continuous with respect to $t$.
\end{lemma}
\begin{proof}
Since $B\subset C(I,X)$ is bounded, $B_0\equiv\{x(t):x\in B, t\in I\}$ is a bounded subset of $X$. Let $r=\|f_{B_0}\|_{C(K)}$, where $f_{B_0}=\mathbb T(B_0)=TQJ(B_0)$, $C(K)$, $T,Q$ and $J$ are defined as in Section 2. It follows from Theorem 3.11 that $\digamma(f_B)=\mu(B)$ for $B\in \mathscr B(X)$ defines a monotone continuous convex function on $V$, and it is $L_B$-Lipschitzian on $V\cap rB_{C(K)}$, where $L_B=\digamma\Big((1+r)f_{B_X}\Big)$. Note that for all $t\in I$, $B(t)\subset B_0$. Since the three mappings $T,Q$ and $J$ are order preserving, we conclude that $f_{B(t)}\leq f_{B_0}$, hence, $f_{B(t)}\in V\cap rB_{C(K)}$.  $\forall s,t\in I$, we have
\begin{equation}\nonumber
\begin{aligned}
|\;\mu(B(t))-\mu(B(s))\;|&=|\;\digamma(f_{B(t)})-\digamma(f_{B(s)})\;|\\
&\leq L_B\cdot\|f_{B(t)}-f_{B(s)}\|\\
&= L_B\cdot\|TQJ(B(t))-TQJ(B(s))\|\\
&= L_B\cdot\|Q[J(B(t))-J(B(s))]\|\\
&\leq L_B\cdot\|J(B(t))-J(B(s))\|\\
&= L_B\cdot d_H(\;\overline{\rm co}(B(t)),\overline{\rm co}(B(s))\;)\\
&\leq L_B\cdot d_H(B(t),B(s))\\
&\leq L_B\cdot\omega(B,|t-s|).
\end{aligned}
\end{equation}
Thus,  the proof is complete.
\end{proof}

\begin{lemma}
 Suppose that $\mu$ is a convex measure of noncompactness on $X$, $B_0\subset C(I,X)$ is a nonempty equicontinuous bounded set, and that $B_{n+1}\subset B_n\neq\emptyset\;(n=0,1,2,\cdots)$. Then $\{\mu(B_n)\}_{n\geq 1}\subset C(I,\R)$ is again an equicontinuous subset.
\end{lemma}

\begin{proof}
Since $B_0$ is an equicontinuous bounded set, each $B_n\neq\emptyset$ is again an equicontinuous bounded subset. Let $C_1:=\{x(t):x\in B_0, t\in I\}$, $r=\|f_{C_1}\|$, where $f_{C_1}=TQJ(C_1)$. Then it follows from Lemma 7.1, for all $ n\in \N$ and $t,s\in I$,
$$|\;\mu(B_n(t))-\mu(B_n(s))\;|\leq L_n\cdot \omega (B_n,|t-s|),$$
where $L_n=\digamma\Big((1+\|f_{B_n(I)}\|)f_{B_X}\Big)$, $B_n(I)=\{x(t):x\in B_n,t\in I\}$. Since $B_n(I)\subset C_1$, $f_{B_n(I)}\leq f_{C_1}$.
According to the monotonicity of $\digamma$, we have $L_{n}\leq\digamma\Big((1+\|f_{C_1}\|)f_{B_X}\Big)$.
\begin{equation}
\begin{aligned}
|\;\mu\Big(B_n(t)\Big)-\mu\Big(B_n(s)\Big)\;|&\leq \digamma\Big((1+\|f_{C_1}\|)f_{B_X}\Big)\cdot \omega(B_n,|t-s|)\\
&\leq\digamma\Big((1+r)f_{B_X}\Big)\omega(B_0,|t-s|).
\end{aligned}
\end{equation}
Since $B_0$ is equicontinuous, $\big\{\mu(B_n(t))\big\}_{n\geq 1}$ is equicontinuous with respect to $t\in I$.
\end{proof}
The following result is an extension of J. Bana\'{s} and K. Goebel \cite[Th 13.3.1]{bag}.

\begin{theorem}
Let $X$ be a Banach space, $I=[0,a]$ and $\mu$ be a convex MNC on $X$, $B(x_0,r)=\{x\in X: \|x-x_0\|\leq r\}$. Suppose that $f(\cdot,\cdot):[0,a]\times B(x_0,r)\rightarrow X$ is uniformly continuous and $w=w(\cdot,\cdot)$ is a Kamke function. If $f$ is bounded on $[0,a]\times B(x_0,r)$ by $(0\leq)\; \xi\leq r$ and if the following condition is satisfied: For all $B\in\mathscr{B}(X)$ and almost all $t\in [0,a_1]$
\begin{equation}\mu(f(t,B))\leq w(t,\mu(B)), \end{equation}
 then the Cauchy problem (7.1)
has at least one solution $x\in C([0,a_1],X)$, where $a_1=\min\{1,a\}$. In particular, if $\mu$ is a sublinear MNC, then $a_1=a$, i.e.
the Cauchy problem (7.1)
has at least one solution $x\in C([0,a],X)$
\end{theorem}

\begin{proof} Let $I=[0,a]$ and
\begin{equation}\nonumber
B_0=\Big\{\;x(t)\in C(I,X):x(0)=x_0\;{\rm with}\; \|x(t)-x(s)\|\leq\xi|t-s|,\;\forall\;s,t\in I\Big\}.
\end{equation}
Then $B_0$ is an equicontinuous bounded closed convex set. The transformation $A: B_0\rightarrow C(I,X)$ defined for $x\in B_0$ and $t\in I$ by
\begin{equation}
(Ax)(t)=x_0+\int_0^t f(s,x(s))ds
\end{equation}
is continuous self-mapping on $B_0$. Let $B_{n+1}=\overline{{\rm co}}\Big(A(B_n)\Big),\;n=0,1,2,\cdots$. It is easy to observe that $\{B_n\}$ is a decreasing sequence of equicontinuous closed convex subsets.  Let $u_n(t)=\mu(B_n(t))$. Obviously $0\leq u_{n+1}(t)\leq u_n(t)$ for $n=0,1,\cdots$. By Lemma 7.2, $\{u_n\}_{n\geq 1}$ is a bounded equicontinuous subset of $C(I,X)$. Therefore, it converges uniformly to a function $u_\infty$ defined by \[u_\infty(t)=\lim\limits_{n\rightarrow \infty}u_n(t)\] for all $t\in I$. Let $y(t)=x_0+f(0,x_0)t$ for all $t\in I$. Then for every $x\in A(B_0)$ and for all $t\in I$, we have
\begin{equation}\nonumber
\|x(t)-y(t)\|\leq t\cdot \sup\Big\{\|f(0,x_0)-f(s,x)\|:s\leq t,\|x-x_0\|\leq\xi s\Big\}\equiv ta(t).
\end{equation}
Note that $\lim\limits_{t\rightarrow 0^+}a(t)=0$ and that
\begin{equation}\nonumber
\Big(A(B_0)\Big)(t)\subset B\Big(y(t),ta(t)\Big)=x_0+tf(0,x_0)+ta(t)B(0,1).
\end{equation}
Since $\mu$ is a convex MNC,
\begin{equation}
\begin{aligned}
u_1(t)&=\mu(B_1(t))=\mu\Big(\overline{{\rm co}}\Big(A(B_0)(t)\Big)\Big)\\
&\leq \mu\Big(x_0+tf(0,x_0)+ta(t)B(0,1)\Big)\\
&=\mu\Big(tf(0,x_0)+ta(t)B(0,1)\Big)\\
&=\mu\Big(ta(t)B(0,1)\Big)
\end{aligned}
\end{equation}
Due to the convexity of $\mu$, for all $0\leq t\leq a$ with $ta(t)\leq1$,
\begin{equation}
\mu\Big(ta(t)B(0,1)\Big)\leq ta(t)\mu\Big(B(0,1)\Big).
\end{equation}
Thus, we conclude that \[\lim\limits_{t\rightarrow 0^+}\frac{u_1(t)}{t}=\lim\limits_{t\rightarrow 0^+}a(t)\mu\Big(B(0,1)\Big)=0.\]
 Since $0\leq u_\infty(t)\leq u_n(t)\leq u_1(t)$,  $\lim\limits_{t\rightarrow 0^+}\frac{u_\infty(t)}{t}=0$.

 Note that the uniform continuity of $f(\cdot,\cdot)$ implies that the mapping $F$ defined by $(Fx)(t)=f(t,x(t))$ maps every equicontinuous subset of $B_0$ into an equicontinuous subset of it.

   If either $0\leq t\leq\min\{1,a\}$, or, $\mu$ is a sublinear MNC, then by Corollary 4.9,
\begin{equation}\nonumber
\begin{aligned}
u_{n+1}(t)&=\mu\Big(\overline{{\rm co}}\Big(A(B_n)(t)\Big)\Big)=\mu\Big(\int_0^tf(s,B_n(s))ds\Big) \\
&\leq\int_0^t\mu\Big\{f(s,B_n(s))\Big\}ds\leq\int_0^tw(s,u_n(s))ds,
\end{aligned}
\end{equation}
which implies $u_\infty(t)\leq\int_0^t w(s,u_\infty(s))ds$. Consequently, $u_\infty(t)=0$ for all $0\leq t\leq\min\{1,a\}$.
Since $u_n$ converges to  $u_\infty=0$ uniformly on $I$,  $\lim\limits_{n\rightarrow\infty}\max\limits_{t\in [0,a]}u_n(t)=0$. It follows  from the generalized  Arzela-Ascoli theorem, the set $B_\infty=\cap_{n=1} ^\infty B_n$  is nonempty convex compact set in $C(I, X)$. Obviously,  $A$ is a self-mapping on $B_\infty$. It follows from the Schauder theorem that $A$ has a fixed point $x\in B_\infty$.

\end{proof}

The next result (i.e. the second example) of solvability of the problem (7.1) is a generalization of Theorem 7.3. It is also an extension of Bana\'{s} and Goebel \cite[Th 13.3.1]{bag}.
The space $C_b(\Omega)$  and the mapping $\mathbb J_G$ are the same as in Definition 4.1.
\begin{theorem}
Let $X$ be a Banach space, $I=[0,a]$ and $\mu$ be a convex MNC on $X$. Suppose that

i) $f(\cdot,\cdot):I\times X\rightarrow X$ is bounded and continuous, and $w=w(\cdot,\cdot)$ is a Kamke function.

ii) For any equicontinous set $G\subset C(I,X)$, $\mathbb J_G[0,a]\subset C_b(\Omega)$ is essentially separable in $C_b(\Omega)$; in particular,  if $F\equiv f(t,\cdot)$ satisfying one of the following three conditions:

a) $F$ maps each equicontinuous set  into an equicontinuous set of $C(I, X)$;

b) $F$ maps every equicontinuous set  into an equiregulated set of $R(I, X)$ (see, Definition 5.6 ii) );

c) $F$ maps every equicontinuous set  into a uniformly measurable set of $L_1(I, X)$ (see, Definition 5.10).

iii)  For all $B\in\mathscr{B}(X)$ and almost all $t\in [0,a]$
\begin{equation}\nonumber
\mu(f(t,B))\leq w(t,\mu(B)), \end{equation}
 then the Cauchy problem (7.1)
has at least one solution $x\in C([0,a_1],X)$, where $a_1=\min\{1,a\}$. In particular, if $\mu$ is a sublinear MNC, then $a_1=a$, i.e.
the Cauchy problem (7.1)
has at least one solution $x\in C([0,a],X)$.
\end{theorem}
\begin{proof}

Assume that $f(t,x):[0,a]\times X\rightarrow X$ is bounded by $\xi$, and let
\begin{equation}\nonumber
B_0=\Big\{\;x(t)\in C(I,X):x(0)=x_0\;{\rm with}\; \|x(t)-x(s)\|\leq\xi|t-s|,\;\forall\;s,t\in I\Big\}.
\end{equation}
Then $B_0$ is an equicontinuous bounded closed convex set. The transformation $A: B_0\rightarrow C(I,X)$ defined for $x\in B_0$ and $t\in I$ by
\begin{equation}
(Ax)(t)=x_0+\int_0^t f(s,x(s))ds
\end{equation}
is continuous self-mapping on $B_0$, and it maps every nonempty subset of $C(I,X)$ into an equicontinuous subset. Let $B_{n+1}=\overline{{\rm co}}\Big(A(B_n)\Big),\;n=0,1,2,\cdots$. Then we obtain  a decreasing sequence $\{B_n\}$ of equicontinuous closed convex subsets. If we repeat the procedure of the proof of Theorem 7.3, then we get a decreasing sequence
$\{u_n\}$ of continuous functions $u_n(t)=\mu\Big(B_n(t)\Big)$ and $u_\infty\leq u_n$ for all $n\in\N$ defined by \[u_\infty(t)=\lim\limits_{n\rightarrow \infty}u_n(t)\] for all $t\in I$ such that  $\lim\limits_{t\rightarrow 0^+}\frac{u_\infty(t)}{t}=0$.

Since for each $n\in\N$, $\mathbb J_{B_n}(I)$ is essentially separable in $C_b(\Omega)$, by Theorem 5.2, $\mathbb J_{B_n}: I\rightarrow C_b(\Omega)$ is strongly measurable. If either $0<t\leq\min\{1,a\}$, or, $\mu$ is a sublinear MNC, then it follows from Corollary 4.9 that
\begin{equation}\nonumber
\begin{aligned}
u_{n+1}(t)&=\mu\Big(\overline{{\rm co}}\Big(A(B_n)(t)\Big)\Big)=\mu\Big(\int_0^tf(s,B_n(s))ds\Big) \\
&\leq\int_0^t\mu\Big\{f(s,B_n(s))\Big\}ds\leq\int_0^tw(s,u_n(s))ds,
\end{aligned}
\end{equation}
which implies $u_\infty(t)\leq\int_0^t w(s,u_\infty(s))ds$. Consequently, $u_\infty(t)=0$ for all $0\leq t\leq1$.
Since $u_n$ converges to  $u_\infty=0$ uniformly on $I$,  $\lim\limits_{n\rightarrow\infty}\max\limits_{t\in [0,1]}u_n(t)=0$. It follows  from the generalized  Arzela-Ascoli theorem, the set $B_\infty=\cap_{n=1} ^\infty B_n$  is nonempty convex compact set in $C(I, X)$. Obviously,  $A$ is a self-mapping on $B_\infty$. It follows from the Schauder theorem that $A$ has a fixed point $x\in B_\infty$.

We finish the proof by noting that one of the three particular cases a), b) and c) can always guarantee that the mappings $\mathbb J_{B_n}: I\rightarrow C_b(\Omega)$ are strongly measurable (Theorems 5.3, 5.7 and 5.11).

\end{proof}

%{\bf Acknowledgements}.

\bibliographystyle{amsalpha}

\end{document}